\documentclass[preprint,twoside,11pt]{article}

\usepackage{amsthm}
\usepackage[abbrvbib]{jmlr2e}
\usepackage{amsmath,amsfonts}
\usepackage{hyperref}
\usepackage{color}
\usepackage{srcltx}
\usepackage{etoolbox}
\usepackage{nameref}
\usepackage{mathrsfs}
\usepackage{bbm}
\usepackage{enumitem}
\usepackage{nicefrac}
\usepackage{mathtools}
\usepackage{dsfont,mathrsfs}
\usepackage{enumerate}
\usepackage{cleveref}
\usepackage{bm}
\usepackage{xargs}
\usepackage{textcomp}

\numberwithin{equation}{section}

\newtheorem{Co}{Corollary}
\newtheorem{Lem}{Lemma}
\newtheorem{Prop}{Proposition}
\newtheorem{Th}{Theorem}

\newtheorem*{prop:gen_conc}{Proposition~\ref{prop:gen_conc}}
\newtheorem*{prop:discr_conc}{Proposition~\ref{prop:discr_conc}}

\theoremstyle{remark} 

\newtheorem{As}{Assumption}
\newtheorem{Rem}{Remark}

\newenvironment{myassump}[2][]
{\renewcommand\theAs{#2}\begin{As}{#1}}{\end{As}}

\Crefname{As}{}{}
\crefname{As}{}{}

\def\argmax{\operatornamewithlimits{argmax}}
\def\argmin{\operatornamewithlimits{argmin}}

\renewcommand\leq\leqslant
\renewcommand\geq\geqslant
\newcommand\eps\varepsilon

\newcommand{\bgamma}{{\boldsymbol{\gamma}}}

\newcommand{\E}{\mathsf E}

\newcommand{\R}{\mathbb R}

\newcommandx{\norm}[2][1=]{\ifthenelse{\equal{#1}{}}{\left\Vert #2 \right\Vert}{\left\Vert #2 \right\Vert^{#1}}}
\newcommand{\AAA}{\mathsf{A}}
\newcommand{\BBB}{\mathsf{B}}

\newcommand{\Kxnorm}{\mathcal{K}}
\newcommand{\fone}{f^{(1)}}
\newcommand{\ftwo}{f^{(2)}}
\newcommand{\Wone}{W^{(1)}}
\newcommand{\Wtwo}{W^{(2)}}
\newcommand{\vone}{v^{(1)}}
\newcommand{\vtwo}{v^{(2)}}
\newcommand{\requ}{\sigma}

\newcommand{\diagone}{\mathsf{\Delta}^{(1)}}
\newcommand{\diagtwo}{\mathsf{\Delta}^{(2)}}

\newcommand{\lbeta}{\lfloor\beta\rfloor}

\newcommand{\A}{\mathcal A}
\newcommand{\B}{\mathcal B}

\newcommand{\D}{\mathcal D}

\newcommand{\G}{\mathcal G}
\renewcommand{\H}{\mathcal H}
\newcommand{\I}{\mathcal I}
\newcommand{\J}{\mathcal J}

\newcommand{\NN}{\mathsf{NN}}

\newcommand{\U}{\mathcal U}

\def\pstar{\mathsf{p}^*}
\def\sfp{\mathsf{p}}
\def\sfq{\mathsf{q}}

\def\Xset{\mathsf{X}}
\def\Yset{\mathsf{Y}}
\def\Wset{\mathsf{W}}
\def\Tset{\mathsf{\Theta}}

\def\Lipgen{\mathsf{L}_2}
\newcommand\sfl{\mathsf{L}}

\def\PP{{\rm P}}

\newcommand{\rmd}{\mathrm{d}}
\newcommand{\Id}{\mathrm{I}}

\newcommand{\Var}{\text{\rm Var}}

\DeclareMathOperator{\JS}{\mathrm{JS}}
\DeclareMathOperator{\KL}{\mathrm{KL}}

\ShortHeadings{Rates of convergence for density estimation with GANs}{Puchkin, Samsonov, Belomestny, Moulines, Naumov}

\firstpageno{1}

\begin{document}

\title{Rates of convergence for density estimation\\ with generative adversarial networks}

\author{\name Nikita Puchkin \email npuchkin@hse.ru \\
\addr HSE University, Russian Federation\\Institute for Information transmission Problems, Russian Federation
\AND
\name Sergey Samsonov \email svsamsonov@hse.ru \\
\addr HSE University, Russian Federation\\Institute for Information transmission Problems, Russian Federation
\AND
\name Denis Belomestny \email denis.belomestny@uni-due.de \\
\addr Duisburg-Essen University, Germany\\
HSE University, Russian Federation
\AND
\name Eric Moulines \email eric.moulines@polytechnique.edu \\
\addr \'{E}cole Polytechnique, France\\
Mohamed Bin Zayed University of AI, United Arab Emirates
\AND
\name Alexey Naumov \email anaumov@hse.ru \\
\addr HSE University, Russian Federation\\Steklov Mathematical Institute of Russian Academy of Sciences, Russian Federation
}

\maketitle

\begin{abstract}%
	In this work we undertake a thorough study of the non-asymptotic properties of the vanilla generative adversarial networks (GANs). We prove an oracle inequality for the Jensen-Shannon (JS) divergence between the underlying density $\pstar$ and the GAN estimate with a significantly better statistical error term compared to the previously known results. The advantage of our bound becomes clear in application to nonparametric density estimation. We show that the JS-divergence between the GAN estimate and $\pstar$ decays as fast as $(\log{n}/n)^{2\beta/(2\beta + d)}$, where $n$ is the sample size and $\beta$ determines the smoothness of $\pstar$. This rate of convergence coincides (up to logarithmic factors) with minimax optimal for the considered class of densities.
\end{abstract}


\section{Introduction}

Let $X_1, \dots, X_n$ be i.i.d. random elements with values in $\Xset \subseteq \R^d$ drawn from a distribution $P^*$. We assume that $P^*$ admits a density $\pstar$ with respect to a dominating measure $\mu$. The measure $\mu$ 
is not necessarily absolutely continuous with respect to the Lebesgue measure, it can be the counting measure or the Hausdorff measure on a low-dimensional manifold as well. Our goal is to estimate $\pstar$ based on a finite sample. The problem of density estimation was extensively studied in the  literature and encounters numerous approaches such as kernel (see, e.g., \citep[Section 1.2]{tsybakov2008introduction} and \citep{mcdonald2017minimax}) and k-nearest neighbors density estimators \citep{dasgupta14}, wavelet thresholding \citep{donoho96}, and aggregation  \citep{rakhlin05, rigollet06, bunea07, b16, dalalyan17}. Recently, \cite{goodfellow2014generative} have introduced a novel approach for a related problem of generative modeling called generative adversarial networks (or simply GANs). A generative adversarial network consists of a generator and a discriminator. Given a known easy-to-sample distribution on a latent space $\Yset$ with a density 
$\phi$, the generator $g: \Yset \rightarrow \Xset$ takes i.i.d. samples $Y_1, \dots, Y_n $ from \(\phi\) and produces fake ones $g(Y_1), \dots, g(Y_n)$. The goal of the discriminator $D$ is to distinguish between the real samples $X_1, \dots, X_n$ and $g(Y_1), \dots, g(Y_n)$. Usually, $D$ can be thought of as a map $\Xset \mapsto (0, 1)$, where larger values correspond to  higher confidence that an input variable is drawn from $\pstar$. In practice, both generator and discriminator usually belong to some parametric families (for example, to classes of neural networks). Let us fix positive integers $d_\G$ and $d_\D$ and some compact sets $\Wset \subset \R^{d_\G}$ and $\Tset \subset \R^{d_\D}$. In our paper we assume that
\[
	g \in \G = \{g_w : w \in \Wset\}
	\quad \text{and} \quad
	D \in \D = \{D_\theta : \theta \in \Tset\}.
\]
As a byproduct of the described generative approach, GANs also provide an implicit density estimate for $\pstar$. Indeed, if a statistician manages to find a good generator $g_w$, then the density of $g_w(Y)$ is a reasonable estimate of $\pstar$. In \citep{goodfellow2014generative}, the authors suggested to solve the following minimax problem, also called \emph{vanilla GAN}:
\begin{equation}
	\label{GAN:Goodfellow}
	\min_{w \in \Wset} \: \max_{\theta \in \Tset}\:
	\biggl\{L(w, \theta)
	:= \frac{1}{2}\E_{X \sim \pstar} \log D_\theta(X)
	+ \frac{1}{2}\E_{Y \sim \phi} \log\bigl( 1 - D_\theta (g_w(Y))\bigr)
	\biggr\}.
\end{equation}
Here and further in the paper, $\log$ stands for the natural logarithm.
The intuition behind \eqref{GAN:Goodfellow} is that if $\D$ contained all measurable functions on $\Xset$ with values in $(0,1)$, the minimax problem \eqref{GAN:Goodfellow} would reduce to (see \cite[Theorem 1]{goodfellow2014generative})
\begin{equation}
	\label{GAN: Goodfellow_distributonal}
	\min_{w \in \mathcal{\Wset}} \big[ \JS(\sfp_w, \pstar) - \log 2 \big],
\end{equation} 
where \(\JS\) is the Jensen-Shannon divergence (see \eqref{eq:JS_divergence_def} for the  definition), and $\sfp_w$ is the density of $g_w(Y)$ with $Y \sim \phi$. Unfortunately, since $\D$ is a parametric class (and, hence, it cannot contain all the measurable functions), the actual value of $\max_{\theta \in \Tset} L(w, \theta)$  differs from $(\JS(\sfp_w, \pstar) - \log 2)$. Fortunately, the gap between $\max_{\theta \in \Tset} L(w, \theta)$ and $(\JS(\sfp_w, \pstar) - \log 2)$ may be rather small for a proper class of discriminators. Since the true distribution $\pstar$ in \eqref{GAN:Goodfellow} is unknown,  we consider a plug-in estimate $\sfp_{\widehat w}$, where $\widehat w$ is a solution of the optimization problem
\begin{equation}
	\label{gan}
	\widehat w
	\in \argmin\limits_{w \in \Wset} \max\limits_{\theta \in \Tset} L_n(w, \theta), 
\end{equation}
and
\begin{equation}
	\label{L_n}
	L_n(w, \theta)
	= \frac1{2n} \sum\limits_{i=1}^n \log D_\theta(X_i) + \frac1{2n} \sum\limits_{j=1}^n \log \Big( 1 
	- D_\theta\big( g_w(Y_j) \big) \Big)
\end{equation}
is the empirical version of the functional $L(w, \theta)$ defined in \eqref{GAN:Goodfellow}.

GANs provide a flexible tool for sampling from an unknown distribution, and they have recently become extremely popular among practitioners. Using deep neural network classes $\mathcal{G}$ and $\mathcal{D}$, one can reach the state-of-the-art generative performance in many challenging tasks, including image super-resolution \citep{ledig:2017:super}, video synthesis \citep{2018:Kim:deep_video}, and many others. Various GAN formulations were later proposed by varying the divergence measure in \eqref{GAN: Goodfellow_distributonal}. For instance, f-GAN \citep{nowozin2016f} generalized vanilla GAN by minimizing a general f-divergence; Wasserstein GAN (WGAN) \citep{arjovsky2017wasserstein} considered the first-order Wasserstein (Kantorovich)  distance ($W_1$ distance); MMD-GAN \citep{dziugaite2015training} was based on the maximum mean discrepancy; energy-based GAN \citep{zhao2016energy} minimizes the total variation distance as discussed in 
\citep{arjovsky2017wasserstein}; Quadratic GAN \citep{feizi2020understanding} aimed to find the distribution minimizing the second-order Wasserstein (Kantorovich) distance.

The empirical success of GANs motivated  many researchers to analyze their theoretical properties.
For example, in \citep{biau2020some, schreuder2020statistical}, the authors carried out theoretical analysis of WGANs.
In \citep{biau2020some}, the authors obtained upper bounds for the excess risk of GANs for parametric classes (including the class of neural networks).
In contrast to \cite{biau2020some}, \cite{schreuder2020statistical} considered generative models based on $\beta$ times differentiable transformations of the $d$-dimensional unit hypercube  and derived rates of convergence of order \(O\left(n^{-\beta/d}\vee n^{-1/2}\right)\) for the corresponding $W_1$ distance.
\cite{liang2018well} used results from the empirical process theory to prove upper bounds for Sobolev GANs (i.e., when generators and discriminators belong to Sobolev classes $\mathcal W^\alpha$ and $\mathcal W^\beta$, respectively), MMD GANs, WGANs, and vanilla GANs.
In the case of Sobolev GANs, the obtained rate $n^{-(\alpha + \beta)/(2\beta + d)}\vee n^{-1/2}$ for the corresponding integral probability metric (IPM) is shown to be minimax optimal.
The research of \cite{liang2018well} was continued in the works of \cite{singh2018minimax, uppal2019, luise2020generalization, chen22, vardanyan23} where the authors studied the performance of GANs in terms of different IPM losses and the Sinkhorn divergence.
The vanilla GANs were studied in \citep{liang2018well, biau2018some, asatryan2020convenient}.
However, the rates of convergence for them in terms of Jensen-Shannon divergence are not yet fully understood.
\cite{biau2018some} and \cite{asatryan2020convenient} improved approximation terms as compared to \cite{liang2018well}, but it is not clear whether the rates obtained in \citep{biau2018some, asatryan2020convenient} are minimax optimal.
In this work, we provide a refined analysis of the theoretical properties of vanilla GANs and derive minimax optimal rates.

\paragraph*{Contributions.}
Our contributions can be summarized as follows.

\begin{itemize}[leftmargin=5mm, noitemsep]
	\item We prove (\Cref{fast_rate}) a sharp oracle inequality for the case when the classes $\G$ and $\D$ are general parametric classes, which significantly improves  the existing inequalities from the works of \cite{biau2018some} and \cite{asatryan2020convenient}.
	\item We apply the result of \Cref{fast_rate} to a nonparametric density estimation problem.
	Choosing $\G$ and $\D$ as classes of neural networks of appropriate architectures with ReQU activation functions, we  derive the rates of convergence for the estimate $\sfp_{\widehat w}$ to 
	the true density $\pstar$ in terms of the Jensen-Shannon divergence.
	Namely, we show that, with probability at least $1 
	- \delta$, it holds that
	\begin{equation}
		\label{eq:final-rate}
		\tag{\Cref{thm:requ}}
		\JS(\sfp_{\widehat w}, \pstar)
		\lesssim \left( \frac{\log n}n \right)^{\frac{2\beta}{2\beta + d}} + \frac{\log (1/\delta)}n,
	\end{equation}
	provided that the actual density \(\pstar\) is the density of a random variable 
	\(g^*(Y)\) with $Y$ being uniformly distributed on \([0, 1]^d\), for a smooth  invertible transform \(g^*\) not necessary belonging to 
	\(\G\). We also discuss that previously known bounds do not yield this rate of convergence.
    \item We show that the result of \Cref{thm:requ} is minimax optimal up to a logarithmic factor. Namely, we prove that for any estimate $\widehat\sfp$, it holds that
	\begin{equation}
		\tag{\Cref{th:lower_bound}}
		\sup\limits_{\pstar} \E \JS(\widehat\sfp, \pstar)
		\gtrsim n^{-\frac{2\beta}{2\beta + d}},
	\end{equation}
	where the supremum is taken with respect to the densities $\pstar$ satisfying the 
	same regularity assumptions as in \Cref{thm:requ}. Hence, our results imply the minimax optimality of GANs (up to logarithmic factors) in the context of nonparametric density estimation. 
\end{itemize}

\section{Preliminaries and notations}
\label{sec:notations}

\paragraph*{Kullback-Leibler and Jensen-Shannon divergences.}
Let $\Omega\subset \R^d$ be a bounded domain. For two probability measures on a 
measurable space $(\Omega, \mathcal B(\Omega))$ with Lebesgue densities $\sfp$ and $\sfq$, respectively, we define the Kullback-Leibler divergence between them as
\[
	\KL(\sfp, \sfq) =
	\begin{cases}
		\int \sfp(x) \log \big (\sfp(x)/\sfq(x)\big) \rmd \mu, \quad \text{if $\sfp \ll \sfq$},\\
		+\infty, \quad \text{otherwise}.
	\end{cases}
\]
Here and further in this paper, $\log$ stands for the natural logarithm. By $\JS(\sfp, \sfq)$, we denote the Jensen-Shannon divergence
\begin{equation}
\label{eq:JS_divergence_def}
\JS(\sfp, \sfq) = \frac12 \KL\left(\sfp, \frac{\sfp+\sfq}2 \right) + \frac12 \KL\left(\sfq, \frac{\sfp+\sfq}2 
	\right).
\end{equation}

\paragraph*{Norms.}
For a matrix $A$ and a vector $v$, we denote by $\|A\|_\infty$ and $\|v\|_\infty$  the maximal 
absolute value of entries of $A$ and $v$, respectively.
$\|A\|_0$ and $\|v\|_0$ shall stand for the number of non-zero entries of $A$ and $v$, 
respectively.
Finally, the Frobenius norm and operator norm of $A$ are denoted by $\|A\|_F$ and 
$\|A\|$, respectively, and the Euclidean norm of $v$ is denoted by $\|v\|$. For $x \in \R^d$ and $r > 0$ we write $\B(x,r) = \{y \in \R^d,\, \|y - x\| \leq r\}$. For a function $f : \Omega\to \mathbb{R}^d$, we set 
\begin{align*}
	&
	\|f\|_{L_\infty(\Omega)} = \sup_{x\in \Omega} \|f(x)\|,
	\\&
	\|f\|_{L_2(\Omega)} =  \left\{\int_{\Omega}\|f(x)\|^2\,\rmd \mu\right\}^{1/2},
\end{align*}
and
\[
	\|f\|_{L_2(\sfp, \Omega)} = \left( \int\limits_{\Omega} \|f(x)\|^2 \, \sfp(x) \, \rmd \mu \right)^{1/2}.
\]
Sometimes, we omit the domain $\Omega$ in the notations $L_\infty(\Omega)$, $L_2(\Omega)$, $L_2(\sfp, \Omega)$ and simply write $L_\infty$, $L_2$, and $L_2(\sfp)$, respectively, if there is no ambiguity.

\paragraph*{Smoothness classes.}
For any $s \in \mathbb N$, the function space $C^{s}(\Omega)$ consists of those 
functions 
over the domain $\Omega$ which have partial derivatives up to order $s$ in  $\Omega$, 
and these derivatives are bounded and continuous in $\Omega$. Formally,
\[
	\textstyle{C^s(\Omega) = \big\{ f: \Omega \to \R^m: \quad \|f\|_{C^s} := \max\limits_{ 
	|\bgamma| \leq s} \|D^{\bgamma} f\|_{L_\infty(\Omega)} < \infty \big\},} 
\]
where, for any multi-index $\bgamma = (\gamma_1,\dots,\gamma_d) \in 
\mathbb{N}_0^d$, the partial differential operator $D^{\bgamma}$ is defined as
\[
	D^{\bgamma}f_i = \frac{\partial^{|\bgamma|} f_i}{\partial x_1^{\gamma_1} \cdots 
	\partial x_d^{\gamma_d}}, \quad i \in \{1,\dots, m\}\,, \text{ and } \|D^{\bgamma} f\|_{L_\infty(\Omega)} = \max\limits_{1 \leq i \leq m} \|D^{\bgamma} f_i\|_{L_\infty(\Omega)}\,.
\]
Here we have written $|\bgamma| = \sum_{i=1}^d \gamma_i$ for the order of 
$D^{\bgamma}$.
To avoid confusion between multi-indices and scalars, we reserve the bold font for the former ones.
For the matrix of first derivatives, we use the usual notation $\nabla f = (\partial 
f_i/\partial 
x_j)$ \(i=1,\ldots,m\), $j = 1, \ldots, d$. For a function $\varphi: \R^d \mapsto \R$, $\varphi \in C^2(\Omega)$, we write $\nabla^2 \varphi(x) \in \R^{d \times d}$ for its Hessian at point $x$. For a function $f:\Omega\to\R^m$ and any positive number $0 < \delta \leq 1$, the
\emph{H\"older constant} of order $\delta$ is given by
\begin{equation}\label{beta}
[f]_{\delta}:=\max_{i \in \{1,\ldots,m\}}\sup_{x \not = y\in\Omega}\frac{|f_i(x)-f_i(y)|}{\min\{1, 
\|x-y\|\}^{\delta}}\;.
\end{equation} 
Now, for any $\alpha >0$, we set $s=\lfloor \alpha \rfloor$ and define the \emph{H\"older ball} $\H^{\alpha}(\Omega,H)$ as
\[
\textstyle{
	\H^{\alpha}(\Omega,H) = \big\{ f \in C^s(\Omega): \quad \|f\|_{\H^{\alpha}} := \max \{ 
	\|f\|_{C^s}, \ \max\limits_{|\bgamma| = s} [D^{\bgamma}f]_{\delta} \} \leq H \big\}.}
\]
Note that if \(f\in \H^{1+\beta}(\Omega,H)\) for some \(\beta>0,\) then, for any \(i \in \{1, \ldots, m\}\), $j \in \{1, \ldots, d\}$, it holds that
\[
\left|\frac{\partial f_i(x)}{\partial x_j} - \frac{\partial f_i(y)}{\partial x_j} \right| 
\leq \|f\|_{\H^{1+\beta}} \cdot \|x-y\|^{1\wedge \beta}\leq H\cdot \|x-y\|^{1\wedge \beta} 
\quad \text{for all $x, y \in \Omega$,}
\] 
since $\|f\|_{\H^{\beta_1}} \leq \|f\|_{\H^{\beta_2}} $ for any 
$\beta_2\geq \beta_1.$ We will also write $f\in \H^{\alpha}(\Omega)$ if  $f\in 
\H^{\alpha}(\Omega,H)$ for some $H<\infty.$
We also introduce a class of $\Lambda$-regular functions $\H^\alpha_\Lambda(\Omega, 
H)$, $\Lambda > 1$:
\begin{equation}
\label{eq:H_alpha_lambda_regular}
\H^\alpha_\Lambda(\Omega, H) = \left\{ f \in \H^{\alpha}(\Omega, H) : 
\Lambda^{-2}\Id_{d\times d}\preceq \nabla f(x)^\top \nabla f(x) \preceq\Lambda^2 \Id_{d\times d} \text{ for all $x \in \Omega$} \right\}\,,
\end{equation}
where for symmetric matrices $A,B \in \R^{d \times d}$ we write $A \preceq B$ if $u^\top (B-A)u \geq 0$ for any $u \in \R^d$.

\paragraph*{Neural networks.}
To give a formal definition of a neural network, we first fix an activation function 
$\sigma: \R \rightarrow \R$. For a  vector $v= (v_1,\dots,v_p) \in \R^p$,
we define the shifted activation function $\sigma_v: \R^p \rightarrow \R^p$ as 
\begin{equation*}
	\sigma_{v}(x) = \bigl(\sigma(x_1-v_1),\dots,\sigma(x_p-v_p)\bigr), \quad x = 
	(x_1,\dots,x_p) 
	\in 
	\R^p.
\end{equation*}
Given a positive integer $N$ and a vector $\A = (p_0, p_1, \dots, p_{N+1}) \in \mathbb N^{N+2}$,  a neural network of depth $N+1$ (with $N$ hidden layers) and architecture \(\A\) is  a function of the form
\begin{equation}
	\label{eq:nn}
	f: \R^{p_0} \rightarrow \R^{p_{N+1}}\,, \quad 
	f(x) = W_N \circ \sigma_{v_{N}} \circ W_{N-1} \circ \sigma_{v_{N-1}} \circ \dots \circ W_1\circ\sigma_{v_1} \circ 
	W_0 \circ x,
\end{equation}
where $W_i \in \R^{p_{i+1} \times p_i}$ are weight matrices and $v_i \in \R^{p_i}$ are shift vectors. 
The numbers $p, p_0, \dots, p_{N + 1}$ should not be confused with the density $\sfp(x)$ nor with $\pstar(x)$, which are always displayed in serif.
Next, we introduce a special subclass class of neural networks of depth $N+1$ with architecture $\A$:
\[
	\NN(N, \A) =
	\left\{ \text{$f$ of the form \eqref{eq:nn}} : \|W_0\|_\infty \vee \max\limits_{1 \leq \ell 
	\leq N} \left\{ \|W_\ell\|_\infty \vee \|v_\ell\|_\infty \right\} \leq 1\right\}.
\]
The maximum number of neurons in one layer $\|\A\|_\infty$ is called the width of the neural network. Similarly to \citep{schmidt-hieber2020}, we consider sparse neural networks assuming that only a few weights are not equal to zero. For this purpose, we introduce a class of neural networks of depth $N+1$ with architecture $\A$ and at most $s$ non-zero weights:
\[
	\NN(N, \A, s) =
	\left\{ f \in \NN(N, \A) : \|W_0\|_0 + \sum\limits_{\ell = 1}^{N} \left(\|W_\ell\|_0 + 
	\|v_\ell\|_0 \right) \leq s\right\}.
\]

\section{Theoretical properties of vanilla GANs: a general oracle inequality}

We begin with a sharp oracle inequality for general parametric classes of generators $\G$ and discriminators $\D$. Following \citep{biau2018some}, we impose the next regularity assumptions on generators and discriminators.

\begin{myassump}{AG}
	\label{assu:G}
	For all $g \in \G$, the image of the latent space $\Yset$ is a subset of $\Xset$, that is, $g(\Yset) \subseteq \Xset$. Moreover, for all $y \in \Yset$, the map $w \mapsto g_w(y)$ is Lipschitz on the parameter space $\Wset$ with a constant $\sfl_\G$. That is, for any $y \in \Yset$ and any $u, v \in \Wset$, it holds that
	\[
		\|g_u(y) - g_v(y)\| \leq \sfl_\G \|u - v\|_\infty.
	\]
\end{myassump}

\begin{myassump}{AD}
	\label{assu:D}
	The maps $x \mapsto D_\theta(x)$ and $\theta \mapsto D_\theta(x)$ are Lipschitz on the ambient space $\Xset$ and on the parameter set $\Tset$ with constants $\sfl_\Xset$ and $\sfl_\Tset$, respectively.  More precisely, for any $x, x_1, x_2 \in \Xset$ and any $\theta, \theta_1, \theta_2 \in \Tset$, the following inequalities hold:
	\[
		\left| D_{\theta}(x_1) - D_{\theta}(x_2) \right| \leq \sfl_{\Xset} \|x_1 - x_2\|
		\quad \text{and} \quad
		\left| D_{\theta_1}(x) - D_{\theta_2}(x) \right| \leq \sfl_{\Theta} \|\theta_1 - \theta_2\|_\infty.
	\]
	Moreover, there exist constants $0 < D_{\min} \leq D_{\max} < 1$ such that
	\[
		D_\theta(x) \in [D_{\min}, D_{\max}] \quad \text{for all $x\in \Xset$ and $\theta \in \Tset$}.
	\]
\end{myassump}

We would like to note that the requirement that all functions from $\D$ are bounded away from $0$ and $1$ is 
needed for the $\log D_\theta$ and $\log (1 - D_\theta)$ to be well defined. Similar conditions appear in the literature for aggregation with the Kullback-Leibler loss (for instance, in \citep{ps06, bs07, rig12, b16}). Finally, similarly to \citep{biau2018some}, we require the densities of fake random elements $\sfp_w(x)$, $w \in \Wset$ to fulfil the following property.

\begin{myassump}{A$\sfp$}
	\label{assu:p}
	For all $x \in \Xset$, the map $w \mapsto \sfp_w(x)$ is Lipschitz on $\Wset$ with a constant $\sfl_\sfp$. That is, for any $x \in \Xset$ and any $u, v \in \Wset$, we have
	\[
		|\sfp_u(x) - \sfp_v(x)| \leq \sfl_\sfp \|u - v\|_\infty.
	\]
\end{myassump}

Under Assumptions \ref{assu:G}, \ref{assu:D}, and \ref{assu:p}, we establish the following oracle inequality for the GAN estimate \eqref{gan}.

\begin{Th}
\label{fast_rate}
Assume \ref{assu:G}, \ref{assu:D}, and \ref{assu:p}. Let $\Wset \subseteq [-1, 1]^{d_\G}$ and $\Tset \subseteq [-1, 1]^{d_\D}$. Then, for any $\delta \in (0, 1)$, with probability at least $1 - \delta$, it holds that
	\begin{align}
		\label{eq:th1_bound}
		\JS(\sfp_{\widehat w}, \pstar) - \Delta_\G - \Delta_\D 
		&\notag
		\lesssim \sqrt{ \frac{(\Delta_\G + \Delta_\D)\big[(d_\G + d_\D) \log(2(\sfl_\G \sfl_\Xset \vee \sfl_\Tset \vee \sfl_\sfp \vee 1)n) + \log(8/\delta)\big]}{n}}
		\\&\quad
		+ C_{\ref{eq:c_d_min_d_max}} \cdot \frac{(d_\G + d_\D) \log(2(\sfl_\G \sfl_\Xset \vee \sfl_\Tset \vee \sfl_\sfp \vee 1)n) + \log(8/\delta)}{n}
	\end{align}
	where
	\[
		\Delta_{\G} = \min\limits_{w \in \Wset} \JS(\sfp_w, \pstar),
		\quad
		\Delta_{\D} = \max\limits_{w \in \Wset} \min\limits_{\theta \in \Tset} [\JS(\sfp_w, \pstar) - \log 2 - L(w, \theta)],
	\]
    and
    \begin{equation}
        \label{eq:c_d_min_d_max}
        C_{\ref{eq:c_d_min_d_max}} = \left( \frac{\log(1 / D_{\min})}{D_{\min}^2} + \frac{\log\big(1 / (1 - D_{\max} \big)}{(1 - D_{\max})^2} \right).
    \end{equation}
	Here $\lesssim$ stands for inequality up to an absolute multiplicative constant.
\end{Th}

\begin{Rem}
    Let us recall that we assumed the number of fake samples $m$ equal to the sample size $n$. In general, if $m \neq n$, $n$ should be replaced by $(m \land n)$ in \eqref{eq:th1_bound}. This does not affect the bound much, because usually $m \geq n$.
\end{Rem}

Theoretical properties of vanilla GANs were studied in the works of \cite{liang2018well, biau2018some, asatryan2020convenient}.
The results of \cite{asatryan2020convenient} mainly concern the case of highly smooth generators, so we postpone a comparison with their rates of convergence, and we will return to it after \Cref{thm:requ}.
To our knowledge, the first upper bound on the Jensen-Shannon divergence between the true density $\pstar$ and the vanilla GAN estimate $\sfp_{\widehat w}$ was obtained in \citep[Theorem 13]{liang2018well}.
In \citep{biau2018some}, the authors significantly improved the approximation terms in the oracle inequality of \citep[Theorem 13]{liang2018well}.
The closest result to our Theorem \ref{fast_rate} in the literature is \citep[Theorem 4.1]{biau2018some}, so let us focus on the comparison of these two results.
First, in \citep{biau2018some} the authors assumed that the true density $\pstar$ is bounded away from $0$ and $+\infty$ on its support and that $\sfp_w$ is uniformly bounded over all $w \in \Wset$, while we avoid such requirements in our analysis. Second, the oracle inequality for the JS risk of the vanilla GAN estimate, established in \citep{biau2018some} under similar assumptions is weaker than the bound from \Cref{fast_rate}. The authors of the work \citep{biau2018some} proved that
\begin{equation}
\label{eq:biau_th41}
    \E \JS(\sfp_{\widehat w}, \pstar) - \Delta_\G - \Delta_\D \lesssim \sqrt{\frac{d_\G + d_\D}{n}}.
\end{equation}
One can also use McDiarmid's inequality (see, e.g. \citep[Corollary 
4]{boucheron04concentration}) to transform the in-expectation guarantee \eqref{eq:biau_th41} into a large deviation bound on $\JS(\sfp_{\widehat w}, \pstar)$ of the form
\begin{equation}
    \label{eq:biau_bound}
    \JS(\sfp_{\widehat w}, \pstar) - \Delta_\G - \Delta_\D
    \lesssim \sqrt{\frac{d_\G + d_\D}{n}} + \sqrt{\frac{\log(1/\delta)}n},
\end{equation}
which holds with probability at least $1 - \delta$.
If the classes $\G,$ $\D$  cannot approximate $g^*$  and the respective optimal discriminator with high accuracy, then $\Delta_\G$ and $\Delta_\D$  are of order $1$. In this case, our rates show no improvements over \citep{biau2018some}.
However, in practice one uses  rather expressive classes of deep neural networks for $\G$ and $\D$, and the bound \eqref{eq:th1_bound} can be significantly better than \eqref{eq:biau_bound}. In Section \ref{sec:example}, we give an example parametric families $\G,\D$ such that  $\Delta_\G$ and $\Delta_\D$ tend to $0$ polynomially fast as $n$ goes to $\infty$. Finally, in contrast to \citep{biau2018some}, we specify the dependence of the rate on the constants $\sfl_\G$ and $\sfl_\Tset$, which may be large, especially in the case of wide and deep networks.

To get further insights of the result of \Cref{fast_rate}, let us elaborate on the properties of $\Delta_\D$. 
The next lemma shows that $\Delta_\D$ exhibits quadratic behaviour and its upper bound is closely related to approximation properties of the class of discriminators considered.

\begin{Lem}
	\label{lem:Delta_D_lower_bound}
	Under Assumption \eqref{assu:D}, for any $w \in \Wset$ and any $\theta \in \Tset$, it holds that
	\[
		\JS(\sfp_w, \pstar) - \log 2 - L(w, \theta)
		\geq \left\| \frac{\pstar}{\pstar + \sfp_w} - D_\theta \right\|_{L_2(\pstar + \sfp_w)}^2
	\]
    and
    \begin{equation}
        \label{eq:delta_D_upper_bound}
        \JS(\sfp_w, \pstar) - \log 2 - L(w, \theta)
		\leq \frac{C_{\ref{eq:c_delta}}^2}{(C_{\ref{eq:c_delta}} - 1)^2 D_{\min}(1 - D_{\max})} \left\| \frac{\pstar}{\pstar + \sfp_w} - D_\theta \right\|_{L_2(\pstar + \sfp_w)}^2,
    \end{equation}
    where
    \begin{equation}
        \label{eq:c_delta}
        C_{\ref{eq:c_delta}} = 1 + \sqrt{\frac{D_{\min}}{(1 - D_{\min}) \log(1 / (1 - D_{\max})} \wedge \frac{1 - D_{\max}}{D_{\max} \log(1 / D_{\min})}}.
    \end{equation}
\end{Lem}
The proof of \Cref{lem:Delta_D_lower_bound} is deferred to \Cref{sec:Delta_D_lower_bound_proof}.
\Cref{lem:Delta_D_lower_bound} plays a key role in derivation of faster rates of convergence. It shows that, for each $w \in \Wset$, $L(w, \theta)$ enjoys a similar curvature as the squared loss. This fact remained unnoticed in the literature. Though it is well known that $\JS(\sfp_w, \pstar)$ has a quadratic behaviour with respect to $\sfp_w$ (see, for instance, our \Cref{lem:js_reg} below), this fact alone is not enough to derive the uniform Bernstein-type inequality \eqref{eq:th1_bound}. Only the combination of \Cref{lem:Delta_D_lower_bound} and \Cref{lem:js_reg} leads us to a new, significantly better result.
In particular, if for any $w \in \Wset$ there exists $\theta(w) \in \Tset$ such that $\|D_{\theta(w)} - \pstar / (\pstar + \sfp_w) \|_{L_2(\pstar + \sfp_w)} \leq \eps$, then $\Delta_\D \lesssim \eps^2$. Hence, in this case, \Cref{fast_rate} and the Cauchy-Schwarz inequality immediately yield that
\[
	\JS(\sfp_{\widehat w}, \pstar)-\Delta_\G 
	\lesssim  \eps^2 + \frac{d_\G + d_\D + \log(1/\delta)}{n}
\]
with probability at least $1 - \delta$.
In \citep{biau2018some}, the authors could not exploit the quadratic behaviour of $(\JS(\sfp_w, \pstar) - \log 2 - L(w, \theta))$ properly, and they only proved that
\[
	\JS(\sfp_{\widehat w}, \pstar) - \Delta_\G
	\lesssim \eps^2 + \sqrt{\frac{d_\G + d_\D}{n}} + \sqrt{\frac{\log(1/\delta)}n}
\]
under more restrictive assumptions.
Finally, we would like to note that, in contrary to the remark in \cite[Section 10.1]{singh2018nonparametric}, the oracle inequality in Theorem 1 does not require the density $\pstar$ to be bounded away from zero.

\section{Example: deep nonparametric density estimation}
\label{sec:example}

The goal of this section is to show that GAN estimates achieve minimax rates of convergence in the problem of nonparametric density estimation. From now on, we assume that $\Xset = \Yset = [0, 1]^d$, $\mu$ is the Lebesgue measure in $\R^d$, and the generators $g_w \in \G$ are non-degenerate maps, so the density of the fake samples with respect to $\mu$ is defined correctly. In this setup, if a latent random element $Y$ is drawn according to the density $\phi$ supported on $[0, 1]^d$, then the corresponding density of $g_w(Y)$ is given by
\begin{equation}
	\label{eq:pg}
	\sfp_w(x)=|\det[\nabla g_w(g_w^{-1}(x))]|^{-1} \phi(g_w^{-1}(x)),
	\quad x \in [0, 1]^d.
\end{equation}
For the ease of exposition, we assume that $Y$ is distributed uniformly on $[0, 1]^d$, so that \eqref{eq:pg} simplifies to
\begin{equation}
	\label{eq:pg_uniform}
	\sfp_w(x)=|\det[\nabla g_w(g_w^{-1}(x))]|^{-1},
	\quad x \in [0, 1]^d.
\end{equation}
We also impose a structural assumption on the underlying density $\pstar$, assuming that it is the density of a random element $g^*(Y)$ where $g^* : [0, 1]^d \rightarrow [0, 1]^d$ is a smooth regular map and, as before, $Y$ has a uniform distribution on $[0, 1]^d$.
\begin{myassump}{A$\pstar$}
	\label{assu:AP}
	There exist constants $\beta > 2$, $H^* > 0$, and $\Lambda > 1$ such that $\pstar$ is 
	of the form 
	\[
		\pstar(x) = |\det[\nabla g^*((g^*)^{-1}(x))]|^{-1}, \quad 
		x \in [0, 1]^d,	
	\]
	with $g^* \in \H^{1+\beta}_\Lambda([0, 1]^d, H^*)$.
\end{myassump}

In fact, Assumption \Cref{assu:AP} is not very restrictive and allows for a quite large class of densities.
The celebrated Brenier's theorem \citep{brenier91} implies that, for any density $\sfp$ with a finite second moment, there exists a convex almost everywhere differentiable function $\varphi$ such that $\nabla\varphi(Y) \sim \sfp$ where $Y \sim \U([0,1]^d)$.  
The Caffarelli's regularity theory \citep{caf91, caf92a, caf92b, caf96} extends the Brenier's result in the following way.
If $\sfp$ is bounded away from zero and infinity, $\Omega = \text{supp}(\sfp)$ is convex, and $\sfp$ is in $C^\beta(\text{Int}(\Omega))$, then the potential $\varphi$ is in $C^{\beta + 2}$.

\begin{Rem}
	\Cref{lem:pmin_pmax} yields that $\pstar$ is bounded away from zero and infinity in the considered model. This is a so-called strong density assumption (see, e.g., \citep[Definition 2.2]{audibert07}), widely used in statistics.
\end{Rem}

When applying GANs to the problem of nonparametric density estimation, we shall take $\G$ and $\D$ to be classes of neural networks with ReQU (rectified quadratic unit) activation functions: 
\[
	\sigma^{\mathsf{ReQU}}(x) = (x \vee 0)^2.
\]
While rectified linear unit (ReLU), defined as 
\[
	\sigma^{\mathsf{ReLU}}(x) = x \vee 0,
\]
is the most common choice for the activation functions in neural networks, it is not suitable for our purposes.
The reason is that we want to use neural networks as generators.
The density of the fake random elements $g_w(Y_1), \dots, g_w(Y_n)$ is given by \eqref{eq:pg_uniform} 
and, to enforce differentiability, we use the ReQU activation function.
Besides, a recent result of \citep{belomestny22} on approximation properties of neural networks with ReQU activations can be used to bound $\Delta_\G$ and $\Delta_\D$ from \Cref{fast_rate}. Since the activation function is fixed, we will  write $\sigma(x)$, instead of $\sigma^{\mathsf{ReQU}}(x)$. In this section, we impose the following assumptions on the classes of generators and discriminators.

\begin{myassump}{AG'}
	\label{assu:G'}
	Fix $N_\G, d_\G \in \mathbb N$, and an architecture $\A_\G \in \mathbb N^{N_\G+2}$ 
	with the first and the last component equal to $d$. There are constants $H_{\G} > 0$ and $\Lambda_\G > 1$ such that
	\[
		\G = \G(\Lambda_\G, H_\G, N_\G, \A_\G, d_\G) = \H_{\Lambda_\G}^2([0, 1]^d, H_\G) \cap 
		\NN(N_\G, \A_\G, d_\G).
	\]
	Besides, $g_w([0, 1]^d) \subseteq [0, 1]^d$ for all $g_w \in \G$.
\end{myassump}

\begin{myassump}{AD'}
	\label{assu:D'}
	Fix $N_\D, d_\D \in \mathbb N$, and an architecture $\A_\D \in \mathbb N^{N_\D+2}$ 
	with the first and the last components equal to $d$ and $1$, respectively.
	There are constants $H_{\D} > 0$ and $0 < D_{\min} \leq D_{\max} < 1$ such that
	\[
		\D  = \D(D_{\min}, D_{\max}, H_\D, N_\D, \A_\D, d_\D) = \H^1([0, 1]^d, H_\D) 
		\cap \NN(N_\D, \A_\D, d_\D),
	\]
    and each $D \in \D$ satisfies
	\[
		D(x) \in [D_{\min}, D_{\max}] \subset [0,1] \quad \text{for all $x\in [0, 1]^d$}.
	\]
\end{myassump}
According to Assumption \ref{assu:G'} and the definition of the class $\NN(N_\G, \A_\G, d_\G)$, the generators are parametrized by vectors with at most $d_\G$ components taking values in $[-1, 1]$. Hence, in the context of Section \ref{sec:example}, we have $\Wset = [-1, 1]^{d_\G}$. Similarly, we take $\Tset = [-1, 1]^{d_\D}$.
Before applying \Cref{fast_rate} to the case of parametric families of neural networks, we first check that the conditions of \Cref{fast_rate} are fulfilled. 

\subsection{Towards the rates of convergence: verifying the conditions of \Cref{fast_rate}}
\label{sec:verifying}

Let us start with Assumptions \ref{assu:G} and \ref{assu:D}. To this end we show that the maps $w \mapsto g_w(y)$ and $\theta \mapsto D_\theta(x)$ are Lipschitz on $[-1, 1]^{d_\G}$ and $[-1, 1]^{d_\D}$, respectively.

\begin{Lem}
	\label{Lem:inf_norm_bound}
	Let $N \in \mathbb{N}$ and fix an architecture $\A = (p_0, p_1, \dots, p_{N+1}) \in \mathbb{N}^{N+2}$.
	Let the matrices $\Wone_i, \Wtwo_i \in [-1, 1]^{p_{i+1} \times p_i}$, $0 \leq i \leq N$, and the vectors $\vone_i, \vtwo_i \in [-1, 1]^{p_i}$, $1 \leq i \leq N$, be such that
	\[
		\norm{\Wone_{i} - \Wtwo_{i}}_{\infty} \leq \eps
		\quad \text{for all $i \in \{0,\dots,L\}$}
	\]
	and
	\[
		\norm{\vone_{i} - \vtwo_{i}}_{\infty} \leq \eps
		\quad \text{for all $i \in \{1, \dots, L\}$.}	
	\]
	Then the neural networks
	\[
		\fone(x) = \Wone_N \circ \requ_{\vone_{N}} \circ \Wone_{N-1} \circ \requ_{\vone_{N-1}} \circ \dots \circ \requ_{\vone_1} \circ \Wone_0 \circ x,
	\]
	\[
		\ftwo(x) = \Wtwo_N \circ \requ_{\vtwo_{N}} \circ \Wtwo_{N-1} \circ \requ_{\vtwo_{N-1}} \circ \dots \circ \requ_{\vtwo_1} \circ \Wtwo_0 \circ x
	\]
	satisfy the inequality
	\begin{equation}
		\label{eq:inf_norm_bound}
		\norm{\fone(x) - \ftwo(x)}_{\infty} \leq \eps (N+1) 2^{N} \prod_{\ell=0}^{N}(p_{\ell}+1)^{2^{N}}
		\quad \text{for all $x \in [0, 1]^d$.}
	\end{equation}
\end{Lem}

The proof of \Cref{Lem:inf_norm_bound} is moved to \Cref{sec:inf_norm_bound}. \Cref{Lem:inf_norm_bound} and the fact that $\D \subset \H^1([0, 1]^d, H_\D)$ (due to Assumption \ref{assu:D'}) imply that Assumptions \ref{assu:G} and \ref{assu:D} are fulfilled with the constants
\[
	\sfl_\G = (N_\G+1) 2^{N_\G} \prod\limits_{p \in \A_\G} (p + 1)^{2^{N_\G}},
	\quad
	\sfl_\Tset = (N_\D+1) 2^{N_\D} \prod\limits_{p \in \A_\D} (p + 1)^{2^{N_\D}},
	\quad \text{and} \quad
	\sfl_\Xset = H_\D.
\]
It only remains to check that Assumption \ref{assu:p} is satisfied as well. We do it in two steps. First, we show that if two generators are close to each other with respect to the $\H^1$-norm, then the corresponding densities are close as well.

\begin{Lem}
	\label{lem:pf-pg}
	Assume \Cref{assu:G'} and consider any $u, v \in [-1, 1]^{d_\G}$. Then the corresponding generators $g_u,$ $g_v$ and the densities $\sfp_u,$  $\sfp_v$ fulfill
	\[
		\left\| \sfp_u - \sfp_v \right\|_{L_\infty([0, 1]^d)}
		\leq  \Lipgen \|g_u - g_v\|_{\H^1([0, 1]^d)}
	\]
	with
	\begin{equation}
	\label{eq:const_L_p_g}
		\Lipgen
		= d^{2 + d/2} \Lambda^{3d} (1 + H_\G \Lambda \sqrt d).
	\end{equation}
\end{Lem}

We provide the proof of \Cref{lem:pf-pg} in \Cref{sec:pf-pg_proof}. Finally, let us show that the $\H^1$-norm of $(g_u - g_v)$ scales linearly with the norm of $(u - v)$.
We need a counterpart of \Cref{Lem:inf_norm_bound} for the Jacobi matrices $\nabla \fone(x)$ and $\nabla \ftwo(x)$.
\begin{Lem}
	\label{Lem:inf_norm_bound_derivative}
	Within the notations of Lemma \ref{Lem:inf_norm_bound}, the neural networks $\fone$ and $\ftwo$ satisfy the inequality
	\begin{equation}
		\label{eq:inf_norm_bound_2}
		\norm{\nabla \fone(x) - \nabla \ftwo(x)}_{\infty}
		\leq \eps N(N+1)2^{N+1} \prod\limits_{\ell=0}^{N} (p_{\ell}+1)^{2^{N+1}+1}
		\quad \text{for all $x \in [0, 1]^d$.}
	\end{equation}
\end{Lem}
One can find the proof of \Cref{Lem:inf_norm_bound_derivative} in \Cref{sec:inf_norm_bound_derivative}.
Lemmata \ref{Lem:inf_norm_bound} and \ref{Lem:inf_norm_bound_derivative} immediately yield that, under Assumption \ref{assu:G'}, we have
\[
	\|g_u - g_v\|_{\H^1([0, 1]^d)}
	\leq  \|u - v\|_\infty \cdot N_\G (N_\G + 1)2^{N_\G + 1} \prod\limits_{p \in \A_\G} (p + 1)^{2^{N_\G + 1} + 1}.
\]
Hence, Assumption \ref{assu:p} holds with
\[
	\sfl_\sfp = d^{2 + d/2} \Lambda^{3d} (1 + H_\G \Lambda \sqrt d) N_\G (N_\G + 1)2^{N_\G + 1} \prod\limits_{p \in \A_\G} (p + 1)^{2^{N_\G + 1} + 1}.
\]
Thus, we proved that Assumptions \ref{assu:G'} and \ref{assu:D'} yield Assumptions \ref{assu:G}, \ref{assu:D}, and \ref{assu:p}.

\subsection{Rates of convergence in nonparametric density estimation with GANs}

The discussion in \Cref{sec:verifying} implies that 
\Cref{fast_rate} can be applied to the setting described in the beginning of \Cref{sec:example}. In this section, we go further and provide upper bounds on $\Delta_\G$ and $\Delta_\D$ under assumptions \ref{assu:D'}, \ref{assu:G'}, and \ref{assu:AP}. The key ingredient of our analysis is the recent result of \cite{belomestny22} quantifying the expressiveness of neural networks with ReQU activations. It ensures that for any $f \in \H^{\beta+1}([0,1]^d, H^*)$, $\beta > 1$, and any $\eps > 0,$ there is a neural network with ReQU activation function from the class $\H^{\lfloor\beta\rfloor}([0,1]^d, H^* + \eps)$ which approximates $f$ within the accuracy $\eps$ with respect to the norm in $\H^{\lfloor\beta\rfloor}([0,1]^d)$. We would like to emphasize that, unlike many other results on approximation properties of neural networks, \cite{belomestny22} considers  simultaneous approximation of a smooth function and its derivatives, which is crucial for our purposes. This result allows us to derive the optimal rates of convergence for GAN estimates in the nonparametric density estimation problem.
However, we would like to emphasize that the proof of Theorem 2 works for any choice of activation function, which yields simultaneous approximation of a function and its derivatives by a neural network with weights taking their values in $[-1, 1]$.

\begin{Th}
\label{thm:requ}
    Assume that \ref{assu:AP} holds with $\beta > 2$.
    Choose $H_\G \geq 2 H^*$, $\Lambda_\G \geq 2\Lambda$, $D_{\min} \leq \Lambda^{-d} / (\Lambda^{d} + \Lambda^{-d})]$, $D_{\max} \geq \Lambda^{d} / (\Lambda^{d} + \Lambda^{-d})]$, and $H_\D$ large enough.
	Then there are positive integers $N_\G$, $d_\G$, $N_\D$, $d_\D$ and the architectures $\A_\G \in \mathbb N^{N_\G + 2},$  $\A_\D \in \mathbb N^{N_\D + 2}$ such that   \ref{assu:G'} and \ref{assu:D'} define nonempty sets \(\G\) and \(\D\), respectively. Let us consider the estimator \eqref{gan} with $\Wset = [-1, 1]^{d_\G}$, $\Tset = [-1, 1]^{d_\D},$ \(g_w\in \G\) and \(D_\theta\in \D.\) For any $\delta \in (0, 1)$, with probability at least $1 - \delta$, $\sfp_{\widehat w}$ satisfies the inequality  
	\begin{equation}
		\label{eq:final_rate_new}
		\JS(\sfp_{\widehat w}, \pstar)
		\lesssim \left( \frac{\log n}{n} \right)^{2\beta / (2\beta + d)} + \frac{\log(1/\delta)}n,
	\end{equation}
	provided that $n \geq n_0$ with $n_0$  depending  only on $\beta, d, \Lambda, H^*$, and $H_\G$.
    In \eqref{eq:final_rate_new}, the notation $\lesssim$ stands for an inequality up to a multiplicative constant depending on $d, \beta, H^*, H_\G$, and $H_\D$ only.  
\end{Th}

\begin{Rem}
The dimensions \(d_\G\) and \(d_\D\) define the complexity of the optimization problem in \eqref{gan} and depend on \(n\). It follows from the proof of Theorem~\ref{thm:requ} that 
\begin{eqnarray*}
d_\G\lesssim \left( \frac{n}{\log n} \right)^{d / (2\beta + d)},\quad d_\D\lesssim \left( \frac{n}{\log n} \right)^{d / (2\beta + d)}.
\end{eqnarray*}
Let us note that this dependence of \(d_\G\) and \(d_\D\) on \(n\) can not be  avoided  in general.
\end{Rem}

\Cref{thm:requ} improves the dependence on both $n$ and $\delta$ in the existing bounds on the JS-divergence between the vanilla GAN estimate $p_{\widehat w}$ and $p^*$.
In \citep[Theorem 3.13 and Theorem 4.4]{asatryan2020convenient}, the authors proved that $\E \JS(p_{\widehat w}, p^*)$ decays as fast as $n^{-1/2}$ if $2\beta > d$.
Our results show that in this case the rate of convergence can be much faster.
The reason for suboptimality of \citep[Theorem 4.4]{asatryan2020convenient} is the use of the chaining technique to control the global supremum of the empirical process $L_n(w, \theta) - L(w, \theta)$.
This approach was successfully applied to WGANs (see, e.g., \citep{liang2018well}) but in the case of vanilla GANs, one can prove better upper bounds for the supremum of \(L_n(w, \theta) - L(w, \theta)\) in a local vicinity of the saddle point.
It is also worth mentioning that, if one uses the bound \eqref{eq:biau_th41} from \citep[Theorem 4.1]{biau2018some} and \citep[Theorem 2]{belomestny22} to control $\Delta_\G$ and $\Delta_\D$, he will get a suboptimal rate of convergence
\begin{equation}
    \label{eq:JS-rate}
    \E \JS(\sfp_{\widehat w}, \pstar)
    \lesssim \left( \frac{\log n}n \right)^{2\beta / (4\beta + d)}.
\end{equation}
Finally, the second term  $\log(1/\delta) / n$ in \eqref{eq:final_rate_new} significantly improves  the standard rate $\sqrt{\log(1/\delta) / n}$ which follows from McDiarmid's inequality.

\Cref{thm:requ} also yields an upper bound on the squared $L_2$-distance. Indeed, under the conditions of \Cref{thm:requ}, \Cref{lem:js_reg} and \Cref{lem:pmin_pmax} yield that
\[
    \|\sfp_{\widehat w} - \pstar\|_{L_2(\mu)}^2
    \lesssim \JS(\sfp_{\widehat w}, \pstar)
    \lesssim \left(\frac{\log n}{n} \right)^{2\beta / (2\beta + d)} + \frac{\log(1/\delta)}n.
\]
This coincides (up to logarithmic factors) with the well-known minimax rate of convergence $(\log n / n)^{2\beta / (2 \beta + d)}$ for estimation of a smooth bounded away from zero density under the squared $L_2$-loss.
It turns out that the bound on the JS-divergence from \Cref{thm:requ} is also minimax optimal up to some logarithmic factors, provided that $\pstar$ belongs to the class of densities satisfying Assumption \ref{assu:AP}.

\begin{Th}
	\label{th:lower_bound}
	Let $(X_1, \dots, X_n)$ be a sample of i.i.d. observations generated from a density 
	$\pstar$ satisfying \Cref{assu:AP}.
	Then for any estimate $\widehat\sfp$  of $\pstar$,  that is, measurable function of \(X_1,\ldots,X_n,\)
	it holds that
	\[
		\sup\limits_{\pstar } \E \JS(\widehat\sfp, \pstar)
		\gtrsim n^{-2\beta / (2\beta + d)}
	\]
	with a hidden constant depending on $d$ only.
\end{Th}

The proof of \Cref{th:lower_bound} relies on the van Trees inequality (see \citep[p. 72]{vt68} and \citep{gl95}).
Though under Assumption \Cref{assu:AP} $\JS(\sfp_w, \pstar)$ is equivalent to $\|\sfp_w - \pstar\|_{L_2}^2$ (see \Cref{lem:js_reg} and \Cref{lem:pmin_pmax} below), we would like to emphasize that \Cref{th:lower_bound} does not follow from the existing lower bounds in nonparametric density estimation (see, for instance, \cite[Exercise 2.10]{tsybakov2008introduction}). The reason is that the class of admissible densities in Assumption \Cref{assu:AP} is narrower, than $\H^\beta([0, 1]^d, H_0)$, $H_0 > 0$, due to the additional assumption that $\pstar$ is the density of $g^*(Y)$ for some $g^* \in \H^{1 + \beta}_\Lambda([0, 1]^d, H^*)$. If $\pstar \in \H^\beta([0, 1]^d, H_0)$, then, according to Brenier's theorem, there is $\varphi^* : [0, 1]^d \rightarrow \R^d$, such that $\nabla \varphi^*(Y) \sim \pstar$. However, Caffarelli's regularity theory does not guarantee that $\nabla^2 \varphi^*$ satisfies the condition
\[
    \Lambda^{-2}\Id_{d\times d}\preceq \nabla^2 \varphi^*(x)^\top \nabla^2 \varphi^*(x) \preceq\Lambda^2 \Id_{d\times d},
\]
as required by Assumption \ref{assu:AP}. Without it, the lower bound will be irrelevant. 
Moreover, the theory does not provide a uniform upper bound on the $(\beta + 2)$-th derivative of $\varphi^*$. In contrary, it states that the $(\beta + 2)$-th derivative of $\varphi^*$ can tend to infinity at the border of $[0, 1]^d$. For these reasons, we find it necessary to prove an explicit lower bound in our setup.

\section{Proofs of the main results}
\label{sec:conv}

This section contains the proofs of our main results, Theorems \ref{fast_rate}, \ref{thm:requ}, and \ref{th:lower_bound}. The proof of the upper bounds relies on the uniform high probability bound on $L_n(w, \theta) - L(w, \theta)$, given below.

\begin{Prop}
	\label{prop:uniform_bound}
	Grant Assumptions \ref{assu:G}, \ref{assu:D}, and \ref{assu:p}. Let $\Wset \subseteq [-1, 1]^{d_\G}$ and $\Tset \subseteq [-1, 1]^{d_\D}$. Then, for any $\delta \in (0, 1)$ and any $\eps \in (0, 2]$, with probability at least $1 - \delta$, it holds that
	\begin{align*}
		\left| L_n(w, \theta) - L(w, \theta) \right|
		&
		\leq 3 C_{\ref{eq:eps_net_const}} \eps + \sqrt{\sfl_{\sfp}} \eps^{1/2}
		\\&\quad
        + 4 \sqrt{ \frac{C_{\ref{eq:c_D_var}} \JS(\sfp_{w}, \pstar)
		(d_\G \log(2/\eps) + d_\D \log(2/\eps) + \log(2/\delta))}{2n}}
		\\&\quad
		+ \sqrt{ \frac{C_{\ref{eq:c_D_var}} \Delta(w, \theta)
		(d_\G \log(2/\eps) + d_\D \log(2/\eps) + \log(2/\delta))}{2n}}
		\\&\quad
		+ \frac{2(C_{\ref{eq:c_D_var}} +  C_{\ref{eq:C_D_def}}) (d_\G \log(2/\eps) + d_\D \log(2/\eps) + 
		\log(2/\delta))}{3 n}
	\end{align*}
	simultaneously for all $w \in \Wset$ and $\theta \in \Tset$.
	Here $\Delta(w, \theta) = \JS(\sfp_w, \pstar) - \log 2 - L(w, \theta)$ and the constants $C_{\ref{eq:eps_net_const}}$, $C_{\ref{eq:c_D_var}}$, and $C_{\ref{eq:C_D_def}}$ are defined as follows:
	\begin{equation}
		\label{eq:eps_net_const}
		C_{\ref{eq:eps_net_const}}
		= \frac{\sfl_\G \sfl_\Xset}{2 - 2D_{\max}} + \frac{\sfl_\Tset }{D_{\min} \wedge (1 - D_{\max})},
	\end{equation}
	\begin{align}
		\label{eq:c_D_var}
		C_{\ref{eq:c_D_var}}
		&\notag
		= \left( \frac{\log^2(2 D_{\min})}{1/2 - D_{\min}} \vee 2 \log^2 2 \right)
		+ \left( \frac{\log^2(2 - 2 D_{\max})}{D_{\max} - 1/2} \vee 2 \log^2 2 \right)
		\\&\quad
		+ \left( \frac{\log(e/(2 D_{\min}))}{2 D_{\min}^2}
		+ \frac{\log(e/(2 - 2D_{\max}))}{2(1 - D_{\max})^2} \right),
	\end{align}
	\begin{equation}
		\label{eq:C_D_def}
		C_{\ref{eq:C_D_def}} = \log\left(\frac1{D_{\max}} \vee \frac1{D_{\min}} \right)
		\vee \log\left(\frac1{1 - D_{\max}} \vee \frac1{1 - D_{\min}} \right)
        = \log \frac1{D_{\min} \land (1 - D_{\max})}.
	\end{equation}
\end{Prop}

The rest of this section is organized in the following way. Section \ref{sec:uniform_bound_proof} is devoted to the proof of \Cref{prop:uniform_bound}. Then we prove Theorems \ref{fast_rate}, \ref{thm:requ}, and \ref{th:lower_bound} in Sections \ref{sec:fast_rate_proof}, \ref{sec:requ_proof}, and \ref{sec:lower_bound_proof}, respectively.
In the beginning of these sections, we restate the main results for convenience.

\subsection{Proof of \Cref{prop:uniform_bound}}
\label{sec:uniform_bound_proof}

Let us start with the sketch of the proof. First, we study large deviations of $L_n(w, \theta) - L(w, \theta)$ for a fixed pair $(w, \theta) \in \Wset \times \Tset$. After that, we show that $L_n(w, \theta)$, $L(w, \theta)$, and $\sqrt{\JS(\sfp_w, \pstar)}$ are H\"{o}lder with respect to $w \in \Wset$ and $\theta \in \Tset$. This and the $\eps$-net argument allow us to derive a uniform large deviation bound on $L_n(w, \theta) - L(w, \theta)$. We split the proof into several steps for the sake of readability.
	
	\medskip
	\noindent{\bf Step 1: large deviation bound on $L_n(w, \theta) - L(w, \theta)$.}\quad
	Under the conditions of \Cref{fast_rate}, the functional $L(w, \theta)$ has a kind of curvature inherited from the properties of the cross-entropy. This special structure allows us to obtain tight large deviation bounds on $| L_n(w, \theta) - L(w, \theta) |$. To be more precise, we have the following concentration result.
\begin{Lem}
	\label{lem:One_Delta_Bernstein}
	Fix any $w \in \Wset$ and $\theta \in \Tset$. Assume the conditions of \Cref{fast_rate}. Then, for any $\delta \in (0, 1)$, with probability at least $1 - \delta$, it holds that
	\begin{equation}
		\label{eq:One_Delta_Bernstein}
		\left| L_n(w, \theta) - L(w, \theta) \right|
		\leq \sqrt{ \frac{C_{\ref{eq:c_D_var}} (9 \JS(\sfp_w, \pstar) + \Delta(w, \theta)) \log(2/\delta)}{2n}}
		+ \frac{2 C_{\ref{eq:C_D_def}} \log(2/\delta)}{3 n},
	\end{equation}
	where $\Delta(w, \theta) = \JS(\sfp_w, \pstar) - \log 2 - L(w, \theta)$, and the constants $C_{\ref{eq:c_D_var}}$, $C_{\ref{eq:C_D_def}}$ are defined in \eqref{eq:c_D_var} and \eqref{eq:C_D_def}, respectively.
\end{Lem}
The proof of \Cref{lem:One_Delta_Bernstein} is moved to \Cref{sec:One_Delta_Bernstein_proof}. Our next goal is to convert the concentration inequality \eqref{eq:One_Delta_Bernstein} into a uniform bound on $| L_n(w, \theta) - L(w, \theta)|$.

\medskip
\noindent{\bf Step 2: towards a uniform bound on $L_n(w, \theta) - L(w, \theta)$.}\quad
Let $\eps > 0$ be a parameter to be specified later and let $\Wset_\eps$ and $\Tset_\eps$ be the minimal $\eps$-nets of $\Wset$ and $\Tset$, respectively, with respect to the norm $\|\cdot\|_{\infty}$. It is straightforward to check that the cardinalities of $\Wset_\eps$ and $\Tset_\eps$ satisfy the inequalities
\[
	\left| \Wset_\eps \right| \leq \left( \frac2\eps \right)^{d_\G}
	\quad \text{and} \quad
	\left| \Tset_\eps \right| \leq \left( \frac2\eps \right)^{d_\D}.
\]
Applying \Cref{lem:One_Delta_Bernstein} and the union bound, we conclude that, for any $\delta \in (0, 1)$, with probability at least $1 - \delta$, it holds that
\begin{align}
	\label{eq:eps-net_bound}
	\left| L_n(w, \theta) - L(w, \theta) \right|
	&\notag
	\leq \sqrt{ \frac{C_{\ref{eq:c_D_var}} (9 \JS(\sfp_w, \pstar) + \Delta(w, \theta)) \log(2|\Wset_\eps||\Tset_\eps|/\delta)}{2n}}
	\\&\quad
	+ \frac{2 C_{\ref{eq:C_D_def}} \log(2|\Wset_\eps||\Tset_\eps|/\delta)}{3 n}	
	\\&\notag
	\leq \sqrt{ \frac{C_{\ref{eq:c_D_var}} (9 \JS(\sfp_w, \pstar) + \Delta(w, \theta)) (d_\G \log(2/\eps) + d_\D \log(2/\eps) + \log(2/\delta))}{2n}}
	\\&\quad\notag
	+ \frac{2 C_{\ref{eq:C_D_def}} (d_\G \log(2/\eps) + d_\D \log(2/\eps) + \log(2/\delta))}{3 n}	
\end{align}
simultaneously for all $(w, \theta) \in \Wset_\eps \times \Tset_\eps$. If we show that $L_n(w, \theta)$, $L(w, \theta)$, and $\JS(\sfp_w, \pstar)$ are Lipschitz or at least H\"{o}lder with respect to $w \in \Wset$ and $\theta \in \Tset$, then we can easily extend the bound \eqref{eq:eps-net_bound} from the $\eps$-net $\Wset_\eps \times \Tset_\eps$ to the whole set $\Wset \times \Tset$. The Lipschitzness of $L_n(w, \theta)$ and $L(w, \theta)$ easily follows from Assumptions \Cref{assu:G} and \Cref{assu:D}.
\begin{Lem}
	\label{lem:Delta-Lip}
	Grant Assumptions \ref{assu:D} and \ref{assu:G}.
	Let $w, w_1, w_2 \in \Wset$ and $\theta, \theta_1, \theta_2 \in \Tset$. Then it holds that
	\[
		\left| L_n(w_1, \theta) - L_n(w_2, \theta) \right|
		\leq \frac{\sfl_\G \sfl_\Xset \|w_1 - w_2\|_\infty}{2 - 2D_{\max}} \quad \text{almost surely.}
	\]
	and
	\[
		\left| L_n(w, \theta_1) - L_n(w, \theta_2) \right|
		\leq \frac{\sfl_\Tset \| \theta_1 - \theta_2 \|_\infty}{D_{\min} \wedge (1 - D_{\max})} \quad \text{almost surely.}
	\]
	Moreover, under Assumptions \ref{assu:D} and \ref{assu:G}, we have
	\[
		\left| L(w_1, \theta) - L(w_2, \theta) \right|
		\leq \frac{\sfl_\G \sfl_\Xset \|w_1 - w_2\|_\infty}{2 - 2D_{\max}}
	\]
	and
	\[
		\left| L(w, \theta_1) - L(w, \theta_2) \right|
		\leq \frac{\sfl_\Tset \| \theta_1 - \theta_2 \|_\infty}{D_{\min} \wedge (1 - D_{\max})}
	\]
	for any $w, w_1, w_2 \in \Wset$ and $\theta, \theta_1, \theta_2 \in \Tset$.
\end{Lem}
The proof of \Cref{lem:Delta-Lip} is given in \Cref{sec:Delta-Lip_proof}. The next result plays a crucial role in the analysis of $\JS(\sfp_w, \pstar)$.

\begin{Lem}
	\label{lem:js_reg}
	Let $\sfp$ and $\sfq$ be the densities of two probability distributions with respect to a dominating measure $\mu$. Then it holds that
	\begin{equation}
		\label{eq:js_reg}
		\frac14 \int \frac{(\sfp(x) - \sfq(x))^2}{\sfp(x) + \sfq(x)} \rmd \mu
		\leq \JS(\sfp, \sfq)
		\leq \frac{\log 2}2  \int \frac{(\sfp(x) - \sfq(x))^2}{\sfp(x) + \sfq(x)} \rmd \mu.
	\end{equation}
\end{Lem}

\Cref{lem:js_reg} yields the H\"{o}lderness of $\JS^{1/2}(\sfp_w, \pstar)$ with respect to $w \in \Wset$, provided that Assumption \ref{assu:p} is fulfilled.

\begin{Co}
	\label{co:js_metric_holder}
	Grant Assumption \ref{assu:p}. Then, for any $u, v \in \Wset$, it holds that
	\[
		\left| \sqrt{\JS(\sfp_u, \pstar)} - \sqrt{\JS(\sfp_v, \pstar)} \right|
		\leq \frac{\sqrt{\sfl_\sfp \log 2}}{2^{1/4}} \|u - v\|_\infty^{1/4}
        < \sqrt{\sfl_\sfp} \|u - v\|_\infty^{1/4}.
	\]
\end{Co}

The proofs of \Cref{lem:js_reg} and \Cref{co:js_metric_holder} are moved to \Cref{sec:js_reg_proof} and \Cref{sec:js_metric_holder_proof}, respectively. We are ready to prove a uniform bound on large deviations of $L_n(w, \theta) - L(w, \theta)$.

\medskip
\noindent{\bf Step 3: uniform bound on $L_n(w, \theta) - L(w, \theta)$.}\quad
Let $E$ be the event where \eqref{eq:eps-net_bound} holds. Choose any $(w, \theta) \in \Wset \times \Tset$ and denote the closest to $(w, \theta)$ element of $\Wset_\eps \times \Tset_\eps$ by $(w_\eps, \theta_\eps)$. Then, due to \Cref{lem:Delta-Lip}, the following holds on $E$:
\begin{align*}
	&
	\left| L_n(w, \theta) - L(w, \theta) \right|
	\\&
    \leq \left| L_n(w_\eps, \theta_\eps) - L(w_\eps, \theta_\eps) \right|
	+ 2 \left(\frac{\sfl_\G \sfl_\Xset}{2 - 2D_{\max}} + \frac{\sfl_\Tset }{D_{\min} \vee (1 - D_{\max})}\right) \eps
	\\&
	\leq 2 \left(\frac{\sfl_\G \sfl_\Xset}{2 - 2D_{\max}} + \frac{\sfl_\Tset }{D_{\min} \vee (1 - D_{\max})}\right) \eps
	\\&\quad
	+ \sqrt{ \frac{C_{\ref{eq:c_D_var}} (9 \JS(\sfp_{w_\eps}, \pstar) + \Delta(w_\eps, \theta_\eps)) (d_\G \log(2/\eps) + d_\D \log(2/\eps) + \log(2/\delta))}{2n}}
	\\&\quad
	+ \frac{2 C_{\ref{eq:C_D_def}} (d_\G \log(2/\eps) + d_\D \log(2/\eps) + \log(2/\delta))}{3 n}.	
\end{align*}
According to \Cref{lem:Delta-Lip} and \Cref{co:js_metric_holder},
\[
	\sqrt{\JS(\sfp_{w_\eps}, \pstar)} \leq \sqrt{\JS(\sfp_{w}, \pstar)} + \sqrt{\sfl_\sfp} \eps^{1/4}
\]
and
\begin{align*}
	&
	\sqrt{\Delta(w_\eps, \theta_\eps)}
	= \sqrt{\JS(\sfp_{w_\eps}, \pstar) - \log 2 - L(w_\eps, \theta_\eps)}
	\\&
	\leq \sqrt{2\JS(\sfp_{w}, \pstar) - \log 2 - L(w, \theta) + 2 \sfl_\sfp \eps^{1/2} + C_{\ref{eq:eps_net_const}} \eps} 
	\\&
	\leq \sqrt{\JS(\sfp_{w}, \pstar)} + \sqrt{\Delta(w, \theta)} + \sqrt{2 \sfl_\sfp} \eps^{1/4} + \sqrt{C_{\ref{eq:eps_net_const}} \eps}.
\end{align*}
Thus, on $E$, it holds that
\begin{align*}
	\left| L_n(w, \theta) - L(w, \theta) \right|
	&
	\leq 2 C_{\ref{eq:eps_net_const}} \eps
	+ 4 \sqrt{ \frac{C_{\ref{eq:c_D_var}} \JS(\sfp_{w}, \pstar)
	(d_\G \log(2/\eps) + d_\D \log(2/\eps) + \log(2/\delta))}{2n}}
	\\&\quad
	+ \sqrt{ \frac{C_{\ref{eq:c_D_var}} \Delta(w, \theta)
	(d_\G \log(2/\eps) + d_\D \log(2/\eps) + \log(2/\delta))}{2n}}
	\\&\quad
	+ \sqrt{ \frac{C_{\ref{eq:c_D_var}} \sfl_{\sfp} \eps^{1/2}
	(d_\G \log(2/\eps) + d_\D \log(2/\eps) + \log(2/\delta))}{n}}
	\\&\quad
	+ \sqrt{ \frac{C_{\ref{eq:c_D_var}} C_{\ref{eq:eps_net_const}} \eps
	(d_\G \log(2/\eps) + d_\D \log(2/\eps) + \log(2/\delta))}{2n}}
	\\&\quad
	+ \frac{2 C_{\ref{eq:C_D_def}} (d_\G \log(2/\eps) + d_\D \log(2/\eps) + \log(2/\delta))}{3 n}.
\end{align*}
where $C_{\ref{eq:eps_net_const}}$ is defined in \eqref{eq:eps_net_const}.
Applying the Cauchy-Schwarz inequality to the fourth and the fifth terms, we get the desired bound.

\subsection{Proof of Theorem~\ref{fast_rate}}
\label{sec:fast_rate_proof}

\begin{Th}[restatement of Theorem \ref{fast_rate}]
Assume \ref{assu:G}, \ref{assu:D}, and \ref{assu:p}. Let $\Wset \subseteq [-1, 1]^{d_\G}$ and $\Tset \subseteq [-1, 1]^{d_\D}$. Then, for any $\delta \in (0, 1)$, with probability at least $1 - \delta$, it holds that
	\begin{align*}
		\JS(\sfp_{\widehat w}, \pstar) - \Delta_\G - \Delta_\D 
		&
		\lesssim \sqrt{ \frac{(\Delta_\G + \Delta_\D)\big[(d_\G + d_\D) \log(2(\sfl_\G \sfl_\Xset \vee \sfl_\Tset \vee \sfl_\sfp \vee 1)n) + \log(8/\delta)\big]}{n}}
		\\&\quad
		+ C_{\ref{eq:c_d_min_d_max}} \cdot \frac{(d_\G + d_\D) \log(2(\sfl_\G \sfl_\Xset \vee \sfl_\Tset \vee \sfl_\sfp \vee 1)n) + \log(8/\delta)}{n}
	\end{align*}
	where
	\[
		\Delta_{\G} = \min\limits_{w \in \Wset} \JS(\sfp_w, \pstar),
		\quad
		\Delta_{\D} = \max\limits_{w \in \Wset} \min\limits_{\theta \in \Tset} [\JS(\sfp_w, \pstar) - \log 2 - L(w, \theta)],
	\]
    and
    \[
        C_{\ref{eq:c_d_min_d_max}} = \left( \frac{\log(1 / D_{\min})}{D_{\min}^2} + \frac{\log\big(1 / (1 - D_{\max} \big)}{(1 - D_{\max})^2} \right).
    \]
	Here $\lesssim$ stands for inequality up to an absolute multiplicative constant.
\end{Th}

	Let us introduce $\overline w \in \argmin_{w \in \Wset} \JS(\sfp_{w}, \pstar)$.
	We begin with studying the excess risk
	\begin{equation}
		\label{eq:excess_JS_start}
		\JS(\sfp_{\widehat w}, \pstar) - \Delta_\G
		= \JS(\sfp_{\widehat w}, \pstar) - \min\limits_{w \in \Wset} \JS(\sfp_{w}, \pstar)
		= \JS(\sfp_{\widehat w}, \pstar) - \JS(\sfp_{\overline w}, \pstar),
	\end{equation}
    where $\widehat w$ is given by \eqref{gan}.
	For any $w \in \Wset$, let $\theta_w^*$ denote the parameter, correspoding to the best discriminator in $\D$:
	\[
		\theta_w^* \in \argmax\limits_{\theta \in \Tset} L(w, \theta).
	\]
	Then the definition of $\Delta_\D$ yields that
	\[
		\JS(\sfp_{\widehat w}, \pstar) - \JS(\sfp_{\overline w}, \pstar)
		\leq \Delta_\D + L(\widehat w, \theta^*_{\widehat w}) - L(\overline w, \theta^*_{\overline w}). 
	\]
	Here we used the fact that for any $w \in \Wset$ and any $\theta \in \Tset$ it holds that $L(w, \theta) + \log 2 \leq \JS(\sfp_w, \pstar)$ (see, e.g., \citep[Proposition 1]{goodfellow2014generative} or \citep[Section 2]{biau2018some}) for the explanation.
	The expression in the right-hand side of the last inequality can be rewritten as follows:
	\begin{align}
		\label{eq:risk_js_bound}
		&\notag
		\Delta_\D + \underbrace{\bigl( L(\widehat w, \theta^*_{\widehat w})
		- L_n(\widehat w, \theta^*_{\widehat w}) \bigr)}_{T_1}
		\\&\quad
		+ \underbrace{\bigl( L_n(\widehat w, \theta^*_{\widehat w}) 
		- L_n(\overline w, \theta^*_{\overline w}) \bigr)}_{T_2}
		+ \underbrace{\bigl( L_n(\overline w, \theta^*_{\overline w})
		- L(\overline w, \theta^*_{\overline w}) \bigr)}_{T_3}.       
	\end{align}
    We split the rest of the proof into several steps for convenience.

    \medskip

    \noindent
    \textbf{Step 1: bounds on $T_1$ and $T_3$.}
    \quad
    Let us take
	\begin{align}
        \label{eq:eps_choice}
	    \eps
        &\notag
        =  \frac{C_{\ref{eq:c_D_var}} (d_\G \log(2n) + d_\D \log(2n) + \log(8/\delta))}
		{C_{\ref{eq:eps_net_const}} n}
        \\&\quad
        \wedge \frac{C_{\ref{eq:c_D_var}}^2 (d_\G \log(2n) + d_\D \log(2n) + \log(8/\delta))^2}
		{\sfl_\sfp^2 n^2} \wedge 1.
	\end{align}
    Then $C_{\ref{eq:eps_net_const}} \eps \leq 1/n$, $\sqrt{\sfl_{\sfp} \eps} \leq 1 / n$, and
    \[
        \log(1 / \eps)
        \leq \log n + \log \left( \frac{C_{\ref{eq:eps_net_const}}}{C_{\ref{eq:c_D_var}}} \vee \frac{\sfl_{\sfp}^2}{C_{\ref{eq:c_D_var}}^2} \right)
        \lesssim \log n + \log(\sfl_\G \sfl_\Xset \vee \sfl_\Tset \vee \sfl_\sfp \vee 1).
    \]
	Applying \Cref{prop:uniform_bound} with $\eps$ defined in \eqref{eq:eps_choice},
	we get that, with probability at least $1 - \delta/2$, we simultaneously have
	\begin{align}
		\label{eq:T1_bound}
		T_1
		&\notag
		\lesssim \sqrt{ \frac{(\JS(\sfp_{\widehat w}, \pstar) + \Delta_\D)
		\big[(d_\G + d_\D) \log(2(\sfl_\G \sfl_\Xset \vee \sfl_\Tset \vee \sfl_\sfp \vee 1)n) + \log(8/\delta)\big]}{n}}
		\\&\quad\notag
		+ C_{\ref{eq:c_d_min_d_max}} \cdot \frac{(d_\G + d_\D) \log(2(\sfl_\G \sfl_\Xset \vee \sfl_\Tset \vee \sfl_\sfp \vee 1)n) + \log(8/\delta)}{n}
		\\&
		\leq \sqrt{ \frac{(\JS(\sfp_{\widehat w}, \pstar) - \Delta_\G)
		\big[(d_\G + d_\D) \log(2(\sfl_\G \sfl_\Xset \vee \sfl_\Tset \vee \sfl_\sfp \vee 1)n) + \log(8/\delta)\big]}{n}}
		\\&\quad\notag
		+ \sqrt{ \frac{(\Delta_\G + \Delta_\D)
		\big[(d_\G + d_\D) \log(2(\sfl_\G \sfl_\Xset \vee \sfl_\Tset \vee \sfl_\sfp \vee 1)n) + \log(8/\delta)\big]}{n}}
		\\&\quad\notag
		+ C_{\ref{eq:c_d_min_d_max}} \cdot \frac{(d_\G + d_\D) \log(2(\sfl_\G \sfl_\Xset \vee \sfl_\Tset \vee \sfl_\sfp \vee 1)n) + \log(8/\delta)}{n}
	\end{align}
	and
	\begin{align}
		\label{eq:T3_bound}
		T_3
		&\notag
		\lesssim \sqrt{ \frac{(\Delta_\G + \Delta_\D)
		\big[(d_\G + d_\D) \log(2(\sfl_\G \sfl_\Xset \vee \sfl_\Tset \vee \sfl_\sfp \vee 1)n) + \log(8/\delta)\big]}{n}}
		\\&\quad
		+ C_{\ref{eq:c_d_min_d_max}} \cdot \frac{(d_\G + d_\D) \log(2(\sfl_\G \sfl_\Xset \vee \sfl_\Tset \vee \sfl_\sfp \vee 1)n) + \log(8/\delta)}{n}.
	\end{align}
	Here we used the fact that $\JS(\sfp_{\overline w}, \pstar) = \Delta_\G$ by the definitions of $\overline w$ and $\Delta_\G$. Also, it holds that $\Delta(w, \theta_w^*) \leq \Delta_\D$ due to the definitions of $\theta_w^*$ and $\Delta_\D$. Note that the hidden constants in \eqref{eq:T1_bound} and \eqref{eq:T3_bound} are absolute.

    \medskip

    \noindent
    \textbf{Step 2: a bound on $T_2$.}
    \quad
	It remains to bound $T_2$ in \eqref{eq:risk_js_bound}.
	For any $w \in \Wset$, let us denote
	\[
		\widehat \theta_w \in \argmax\limits_{\theta \in \Tset} L_n(w, \theta).
	\]
	By the definition of $\widehat w$, we have
	\[
		L_n(\widehat w, \widehat \theta_{\widehat w}) \leq L_n(\overline w, \widehat \theta_{\overline w}).
	\]
	Then $T_2 \leq L_n(\overline w, \widehat \theta_{\overline w}) - L_n(\overline w, \theta^*_{\overline w})$, and, applying \Cref{prop:uniform_bound} with $\eps$ from \eqref{eq:eps_choice}
    again, we obtain that, with probability at least $1 - \delta / 2$, it holds that
	\begin{align*}
		T_2
		&
		\lesssim L(\overline w, \widehat \theta_{\overline w}) - L(\overline w, \theta^*_{\overline w})
		\\&\quad
		+ \sqrt{ \frac{(\Delta_\G + \Delta(\overline w, \widehat \theta_{\overline w}))
		\big[(d_\G + d_\D) \log(2(\sfl_\G \sfl_\Xset \vee \sfl_\Tset \vee \sfl_\sfp \vee 1)n) + \log(8/\delta)\big]}{n}}
		\\&\quad
		+ \sqrt{ \frac{(\Delta_\G + \Delta_\D)
		\big[(d_\G + d_\D) \log(2(\sfl_\G \sfl_\Xset \vee \sfl_\Tset \vee \sfl_\sfp \vee 1)n) + \log(8/\delta)\big]}{n}}
		\\&\quad
		+ C_{\ref{eq:c_d_min_d_max}} \cdot \frac{(d_\G + d_\D) \log(2(\sfl_\G \sfl_\Xset \vee \sfl_\Tset \vee \sfl_\sfp \vee 1)n) + \log(8/\delta)}{n},
	\end{align*}
    where, as before, the hidden constants are absolute.
	Since
	\[
		\Delta(\overline w, \widehat \theta_{\overline w})
		= \Delta(\overline w, \theta^*_{\overline w}) + L(\overline w, \theta^*_{\overline w})
		- L(\overline w, \widehat\theta_{\overline w})
		\leq \Delta_\D + L(\overline w, \theta^*_{\overline w})
		- L(\overline w, \widehat\theta_{\overline w}),
	\]
	the following inequality holds on the same event:
	\begin{align*}
		T_2
		&
		\lesssim  - \left( L(\overline w, \theta^*_{\overline w}) - L(\overline w, \widehat \theta_{\overline w}) \right)
		\\&\quad
		+ \sqrt{ \frac{\left(L(\overline w, \theta^*_{\overline w}) - L(\overline w, \widehat\theta_{\overline w}) \right)
		\big[(d_\G + d_\D) \log(2(\sfl_\G \sfl_\Xset \vee \sfl_\Tset \vee \sfl_\sfp \vee 1)n) + \log(8/\delta)\big]}{n}}
		\\&\quad
		+ \sqrt{ \frac{(\Delta_\G + \Delta_\D)
		\big[(d_\G + d_\D) \log(2(\sfl_\G \sfl_\Xset \vee \sfl_\Tset \vee \sfl_\sfp \vee 1)n) + \log(8/\delta)\big]}{n}}
		\\&\quad
		+ C_{\ref{eq:c_d_min_d_max}} \cdot \frac{(d_\G + d_\D) \log(2(\sfl_\G \sfl_\Xset \vee \sfl_\Tset \vee \sfl_\sfp \vee 1)n) + \log(8/\delta)}{n}.
	\end{align*}
	Maximizing the right-hand side over $(L(\overline w, \theta^*_{\overline w})
	- L(\overline w, \widehat\theta_{\overline w}) )^{1/2}$, we obtain that
	\begin{align}
		\label{eq:T2_bound}
		T_2
		&\notag
		\lesssim  \sqrt{ \frac{(\Delta_\G + \Delta_\D)
		\big[(d_\G + d_\D) \log(2(\sfl_\G \sfl_\Xset \vee \sfl_\Tset \vee \sfl_\sfp \vee 1)n) + \log(8/\delta)\big]}{n}}
		\\&\quad
		+ C_{\ref{eq:c_d_min_d_max}} \cdot \frac{(d_\G + d_\D) \log(2(\sfl_\G \sfl_\Xset \vee \sfl_\Tset \vee \sfl_\sfp \vee 1)n) + \log(8/\delta)}{n}.
	\end{align}

    \medskip

    \noindent
    \textbf{Step 3: final bound.}
    \quad
	Consider the union of events where \eqref{eq:T1_bound}, \eqref{eq:T3_bound}, and \eqref{eq:T2_bound} hold. Note that the probability measure of this event is at least $1 - \delta$. Moreover, on this event, we have
	\begin{align*}
		&
        \JS(\sfp_{\widehat w}, \pstar) - \Delta_\G - \Delta_\D 
		\leq T_1 + T_2 + T_3
		\\&
		\lesssim \sqrt{ \frac{(\JS(\sfp_{\widehat w}, \pstar) - \Delta_\G)
		\big[(d_\G + d_\D) \log(2(\sfl_\G \sfl_\Xset \vee \sfl_\Tset \vee \sfl_\sfp \vee 1)n) + \log(8/\delta)\big]}{n}}
		\\&\quad
		+ \sqrt{ \frac{(\Delta_\G + \Delta_\D)
		\big[(d_\G + d_\D) \log(2(\sfl_\G \sfl_\Xset \vee \sfl_\Tset \vee \sfl_\sfp \vee 1)n) + \log(8/\delta)\big]}{n}}
		\\&\quad
		+ C_{\ref{eq:c_d_min_d_max}} \cdot \frac{(d_\G + d_\D) \log(2(\sfl_\G \sfl_\Xset \vee \sfl_\Tset \vee \sfl_\sfp \vee 1)n) + \log(8/\delta)}{n}.
		\\&
		\leq \sqrt{ \frac{(\JS(\sfp_{\widehat w}, \pstar) - \Delta_\G - \Delta_\D)_+
		\big[(d_\G + d_\D) \log(2(\sfl_\G \sfl_\Xset \vee \sfl_\Tset \vee \sfl_\sfp \vee 1)n) + \log(8/\delta)\big]}{n}}
		\\&\quad
		+ 2 \sqrt{ \frac{(\Delta_\G + \Delta_\D)\big[(d_\G + d_\D) \log(2(\sfl_\G \sfl_\Xset \vee \sfl_\Tset \vee \sfl_\sfp \vee 1)n) + \log(8/\delta)\big]}{n}}
		\\&\quad
		+ C_{\ref{eq:c_d_min_d_max}} \cdot \frac{(d_\G + d_\D) \log(2(\sfl_\G \sfl_\Xset \vee \sfl_\Tset \vee \sfl_\sfp \vee 1)n) + \log(8/\delta)}{n}.
	\end{align*}
	Since the inequality $x \leq 2a \sqrt{x} + b$ yields $\sqrt{x} \leq a + \sqrt{a^2 + b}$ and, hence, $x \leq 4a^2 + 2b$, we obtain that
	\begin{align*}
		\JS(\sfp_{\widehat w}, \pstar) - \Delta_\G - \Delta_\D
        &
        \leq \left( \JS(\sfp_{\widehat w}, \pstar) - \Delta_\G - \Delta_\D \right)_+
		\\&
        \lesssim \sqrt{ \frac{(\Delta_\G + \Delta_\D)\big[(d_\G + d_\D) \log(2(\sfl_\G \sfl_\Xset \vee \sfl_\Tset \vee \sfl_\sfp \vee 1)n) + \log(8/\delta)\big]}{n}}
		\\&\quad
		+ C_{\ref{eq:c_d_min_d_max}} \cdot \frac{(d_\G + d_\D) \log(2(\sfl_\G \sfl_\Xset \vee \sfl_\Tset \vee \sfl_\sfp \vee 1)n) + \log(8/\delta)}{n}
	\end{align*}
	on the event of probability measure at least $1 - \delta$.

\subsection{Proof of \Cref{thm:requ}}
\label{sec:requ_proof}

\begin{Th}[restatement of Theorem \ref{thm:requ}]
    Assume that \ref{assu:AP} holds with $\beta > 2$.
    Choose $H_\G \geq 2 H^*$, $\Lambda_\G \geq 2\Lambda$, $D_{\min} \leq \Lambda^{-d} / (\Lambda^{d} + \Lambda^{-d})]$, $D_{\max} \geq \Lambda^{d} / (\Lambda^{d} + \Lambda^{-d})]$, and $H_\D$ large enough.
	Then there are positive integers $N_\G$, $d_\G$, $N_\D$, $d_\D$ and the architectures $\A_\G \in \mathbb N^{N_\G + 2},$  $\A_\D \in \mathbb N^{N_\D + 2}$ such that   \ref{assu:G'} and \ref{assu:D'} define nonempty sets \(\G\) and \(\D\), respectively. Let us consider the estimator \eqref{gan} with $\Wset = [-1, 1]^{d_\G}$, $\Tset = [-1, 1]^{d_\D},$ \(g_w\in \G\) and \(D_\theta\in \D.\) For any $\delta \in (0, 1)$, with probability at least $1 - \delta$, $\sfp_{\widehat w}$ satisfies the inequality  
	\[
		\JS(\sfp_{\widehat w}, \pstar)
		\lesssim \left( \frac{\log n}{n} \right)^{2\beta / (2\beta + d)} + \frac{\log(1/\delta)}n,
	\]
	provided that $n \geq n_0$ with $n_0$  depending  only on $\beta, d, \Lambda, H^*$, and $H_\G$.
    In \eqref{eq:final_rate_new}, the notation $\lesssim$ stands for an inequality up to a multiplicative constant depending on $d, \beta, H^*, H_\G$, and $H_\D$ only. 
\end{Th}

    Applying \Cref{fast_rate}, we obtain that, for any $\delta \in (0, 1)$, the GAN estimate $\sfp_{\widehat w}$ satisfies the inequality
    \begin{align}
        \label{eq:th_fast_rate_applied}
		\JS(\sfp_{\widehat w}, \pstar) - \Delta_\G - \Delta_\D 
		&\notag
		\lesssim \sqrt{ \frac{(\Delta_\G + \Delta_\D)\big[(d_\G + d_\D) \log(2(\sfl_\G \sfl_\Xset \vee \sfl_\Tset \vee \sfl_\sfp \vee 1)n) + \log(8/\delta)\big]}{n}}
		\\&\quad
		+ C_{\ref{eq:c_d_min_d_max}} \cdot \frac{(d_\G + d_\D) \log(2(\sfl_\G \sfl_\Xset \vee \sfl_\Tset \vee \sfl_\sfp \vee 1)n) + \log(8/\delta)}{n}
	\end{align}
    on an event with probability at least $1 - \delta$. Let us recall that
	\[
		\Delta_{\G} = \min\limits_{w \in \Wset} \JS(\sfp_w, \pstar),
		\quad \text{and} \quad
		\Delta_{\D} = \max\limits_{w \in \Wset} \min\limits_{\theta \in \Tset} \big[\JS(\sfp_w, \pstar) - \log 2 - L(w, \theta) \big].
	\]
    In the rest of the proof, we provide upper bounds on the approximation terms $\Delta_\G$ and $\Delta_\D$ and specify the architectures of neural networks and the parameters $d_\G$ and $d_\D$ as well. Our approach relies on the following result, concerning approximation properties of neural networks with ReQU activations.
    
    \begin{Th}[\cite{belomestny22}, Theorem 2]
    \label{prop:requ_approx}
	   Let $\beta > 2$ and let $p, d \in \mathbb N$. Then, for any $H > 0$, $f : [0, 1]^d \rightarrow \R^p$, $f \in \H^\beta([0,1]^d, H)$ and any integer $K \geq 2$, there exists a neural network $h_f : [0, 1]^d \rightarrow \R^p$ of the width
	   \[
            \bigl(4d (K + \lbeta)^d \bigr) \vee 12\,\left((K + 2\lbeta) + 1\right) \vee p
        \]
        with
        \[
            6 + 2(\lbeta-2) + \lceil \log_{2}{d} \rceil + 2\left( \left\lceil \log_2 (2d\lbeta + d) \vee \log_2 \log_2 H \right\rceil \vee 1 \right)
        \]
        hidden layers and at most $p (K + \lbeta)^{d} C(\beta, d, H)$ non-zero weights taking their values in $[-1, 1]$, such that, for any $\ell \in \{0, \dots, \lfloor\beta\rfloor\}$,
	   \begin{equation}
		  \label{eq:approx_H_l_norm}
		  \left\|f - h_f\right\|_{\H^\ell([0, 1]^d)}
		  \leq \frac{ (\sqrt{2}e d)^{\beta} H}{K^{\beta-\ell}} + \frac{9^{d(\lbeta - 1)} (2\lbeta + 1)^{2d + \ell} (\sqrt{2}ed)^\beta H}{K^{\beta - \ell}}.
	   \end{equation}
	   The above constant $C(\beta, d, H)$ is given by
	   \begin{align}
            \label{eq:const_sparsity}
            C(\beta, d, H)
	        &\notag
            = \left( 60 \bigl( \left\lceil \log_2 (2d\lbeta + d) \vee \log_2 \log_2 H \right\rceil \vee 1 \bigr) + 38 \right) 
            \\&\quad
            + 20 d^2 + 144 d \lbeta + 8 d.
	   \end{align}
    \end{Th}
    
    We split the proof of \Cref{thm:requ} into several steps for convenience.

    \bigskip

    \noindent{\bf Step 1: bounding $\Delta_\G$.}\quad
	According to \Cref{lem:js_reg}, it holds that
    \[
        \Delta_\G
        = \min\limits_{w \in \Wset} \JS(\sfp_w, \pstar)
		\leq \min\limits_{w \in \Wset} \frac{\log 2}2 \int \frac{(\sfp_w(x) - \pstar(x))^2}{\sfp_w(x) + \pstar(x)} \rmd \mu.
    \]
    Note that under Assumptions \ref{assu:G'} and \ref{assu:AP}, the densities $\pstar$ and $\sfp_w$ are bounded away from zero and infinity.

    \begin{Lem}
	   \label{lem:pmin_pmax}
	   Let a map $g : [0, 1]^d \rightarrow [0, 1]^d$ be from the class $\H_\Lambda^2([0, 1]^d, H_\G)$ and let
	   \[
		  \sfp(x) = |\det[\nabla g(g^{-1}(x))]|^{-1},
		  \quad x \in [0, 1]^d.
	   \]
	   Then it holds that
	   \[
		  \sfp_{\min} \leq \sfp(x) \leq \sfp_{\max}
		  \quad \text{for any $x \in [0, 1]^d$,}
	   \]
	   where 
	   \begin{equation}
	   \label{eq:p_min_p_max_def}
            \sfp_{\min} = \Lambda^{-d}, \quad
    		\sfp_{\max} = \Lambda^d.
	   \end{equation}
    \end{Lem}

    The proof of \Cref{lem:pmin_pmax} is provided in \Cref{sec:pmin_pmax_proof}. \Cref{lem:pmin_pmax} implies that
    \[
        \Delta_\G \lesssim \Lambda^d \min\limits_{w \in \Wset} \left\| \sfp_w - \pstar \right\|_{L_\infty([0, 1]^d)}^2.
    \]
    Combining this bound with the result of \Cref{lem:pf-pg}, we obtain that
    \[
        \Delta_\G \lesssim \Lambda^{9d} \min\limits_{w \in \Wset} \left\| g_w - g^* \right\|_{\H^1([0, 1]^d)}^2.
    \]
    Let us introduce an auxiliary parameter
	\[
		K = \left\lceil \frac{n}{\log n} \right\rceil^{1 / (2\beta + d)}
	\]
	and set $N_\G$ and $d_\G$ as follows:
	\[
		N_\G = 6 + 2(\lbeta-1) + \lceil \log_{2}{d} \rceil + 2\left( \left\lceil \log_2 (2d\lbeta + 3d) \vee \log_2 \log_2 H^* \right\rceil \vee 1 \right),
	\]
	\begin{equation}
        \label{eq:s_G_bound}
		d_\G
		= d \; C(\beta + 1, d, H^*) (K + \lbeta + 1)^{d}
        \lesssim \left( \frac{n}{\log n} \right)^{d / (2\beta + d)},
	\end{equation}
    where the constant $C(\beta, d, H^*)$ is given by \eqref{eq:const_sparsity}.
    According to \Cref{prop:requ_approx}, there is a neural network $g_w$ of the width
	\[
        \bigl(4d (K + \lbeta + 1)^d \bigr) \vee 12 \left((K + 2\lbeta) + 3\right) \vee d,
    \]
    with ReQU activations, $N_\G$ hidden layers and at most $d_\G$ non-zero weights taking their values in $[-1, 1]$, such that
    \[
        \left\|g^* - g_w \right\|_{\H^1([0, 1]^d)}
        \leq \frac{ (\sqrt{2}e d)^{\beta + 1} H^*}{K^\beta} + \frac{9^{d\lbeta} (2\lbeta + 3)^{2d + 1} (\sqrt{2}ed)^{\beta + 1} H^*}{K^\beta}
        \lesssim \left( \frac{\log n}{n} \right)^{\beta / (2\beta + d)}
    \]
    and
    \[
        \left\|g^* - g_w \right\|_{\H^2([0, 1]^d)} \lesssim \left( \frac{\log n}{n} \right)^{(\beta - 1) / (2\beta + d)}.
    \]
    Hence, it holds that
    \begin{equation}
        \label{eq:delta_G_bound}
        \Delta_\G
        \lesssim \Lambda^{9d} \min\limits_{w \in \Wset} \left\| g_w - g^* \right\|_{\H^1([0, 1]^d)}^2
        \lesssim \Lambda^{9d} \left( \frac{\log n}{n} \right)^{2\beta / (2\beta + d)}.
    \end{equation}
    
    \bigskip
    
	\noindent{\bf Step 2: bounding $\Delta_\D$.}\quad
    According to \Cref{lem:Delta_D_lower_bound}, it holds that
    \[
        \Delta_\D
        \lesssim \max\limits_{w \in \Wset} \min\limits_{\theta \in \Tset} \left\| \frac{\pstar}{\pstar + \sfp_w} - D_\theta \right\|_{L_2(\pstar + \sfp_w)}^2
        \lesssim \max\limits_{w \in \Wset} \min\limits_{\theta \in \Tset} \left\| \frac{\pstar}{\pstar + \sfp_w} - D_\theta \right\|_{L_\infty([0, 1]^d)}^2.
    \]
    Our goal is to show that it is possible to approximate $\pstar / (\pstar + \sfp_w)$ with a neural network, taking its values in $[D_{\min}, D_{\max}]$. 
    Using \eqref{eq:pg_uniform}, \Cref{lem:pmin_pmax}, and the fact that $g_w \in \H^{\beta + 1}(H_\G, [0,1]^d)$, $g^* \in \H^{\beta + 1}(H^*, [0,1]^d)$, it is straightforward to check that, for any $w \in \Wset$,
	\[
		\frac{\pstar}{\pstar + \sfp_w} \in \H^\beta(H^\circ, [0, 1]^d),
	\]
	where $H^\circ$ is a constant, depending on $H_\G, H^*, \beta$, and $d$.
	Let us set the parameters $N_\D$, and $d_\D$ equal to
	\[
		N_\D = 6 + 2(\lbeta - 2) + \lceil \log_{2}{d} \rceil + 2\left( \left\lceil \log_2 (2d\lbeta + d) \vee \log_2 \log_2 H^\circ \right\rceil \vee 1 \right),
	\]
	\begin{equation}
        \label{eq:s_D_bound}
		d_\D
		= d \; C(\beta, d, H^*) (K + \lbeta)^{d}
        \lesssim \left( \frac{n}{\log n} \right)^{d / (2\beta + d)},
	\end{equation}
    where the constant $C(\beta, d, H^\circ)$ is defined in \eqref{eq:const_sparsity}.
    Then, according to \Cref{prop:requ_approx}, for any $w \in \Wset$, there exists $\theta(w) \in \Tset$ and a neural network $D_{\theta(w)}$
    of the width
	\[
        \bigl(4d (K + \lbeta)^d \bigr) \vee 12 \left((K + 2\lbeta) + 1\right),
    \]
    with ReQU activations, $L_\D$ hidden layers and at most $d_\D$ non-zero weights taking their values in $[-1, 1]$, such that
	\begin{align*}
		\left\| \frac{\pstar}{\pstar + \sfp_w} - D_{\theta(w)} \right\|_{L_\infty([0,1]^d)}
		&
        \leq \frac{ (\sqrt{2}e d)^{\beta} H^\circ}{K^\beta} + \frac{9^{d(\lbeta - 1)} (2\lbeta + 1)^{2d} (\sqrt{2}ed)^\beta H^\circ}{K^\beta}
        \\&
        \lesssim \left( \frac{\log n}{n} \right)^{2\beta / (2\beta + d)}.
	\end{align*}
    Thus, we have
    \begin{equation}
        \label{eq:delta_D_bound}
        \Delta_\D
        \lesssim \max\limits_{w \in \Wset} \min\limits_{\theta \in \Tset} \left\| \frac{\pstar}{\pstar + \sfp_w} - D_\theta \right\|_{L_\infty([0, 1]^d)}^2
        \lesssim \left( \frac{\log n}{n} \right)^{2\beta / (2\beta + d)}.
    \end{equation}
	
	\bigskip
	
	\noindent{\bf Step 3: final bound.}\quad
	According to \Cref{Lem:inf_norm_bound} and \Cref{Lem:inf_norm_bound_derivative}, it holds that
    \[
        \log(\sfl_\G \sfl_\Xset \vee \sfl_\Tset \vee \sfl_\sfp \vee 1)
        \lesssim \log (K \Lambda).
    \]
    This and the inequalities \eqref{eq:th_fast_rate_applied}, \eqref{eq:s_G_bound}, \eqref{eq:delta_G_bound}, \eqref{eq:s_D_bound}, and \eqref{eq:delta_D_bound} yield that
	\begin{align*}
		\JS(\sfp_{\widehat w}, \pstar)
		&
		\lesssim \sqrt{ \frac{(\Delta_\G + \Delta_\D)\big[(d_\G + d_\D) \log(2(\sfl_\G \sfl_\Xset \vee \sfl_\Tset \vee \sfl_\sfp \vee 1)n) + \log(8/\delta)\big]}{n}}
		\\&\quad
		+ C_{\ref{eq:c_d_min_d_max}} \cdot \frac{(d_\G + d_\D) \log(2(\sfl_\G \sfl_\Xset \vee \sfl_\Tset \vee \sfl_\sfp \vee 1)n) + \log(8/\delta)}{n}
        \\&
        \lesssim \sqrt{ \frac{(\Delta_\G + \Delta_\D)\big[(d_\G + d_\D) \log(2 K \Lambda} n) + \log(8/\delta)\big]}{n}
		\\&\quad
		+ C_{\ref{eq:c_d_min_d_max}} \cdot \frac{(d_\G + d_\D) \log(2 K \Lambda} n) + \log(8/\delta){n}
        \\&
        \lesssim \Lambda^{9d /2} \left( \frac{\log n}{n} \right)^{\frac{2\beta}{2\beta + d}}
		+ \frac{C_{\ref{eq:c_d_min_d_max}}}n \cdot \left(\log(2 \Lambda n) \left( \frac{n}{\log n} \right)^{\frac d{2\beta + d}} + \log(8/\delta)\right)
		\\&
		\lesssim \left( \Lambda^{9d /2} + C_{\ref{eq:c_d_min_d_max}} \log \Lambda \right) \left( \frac{\log n}{n} \right)^{2\beta / (2\beta + d)} + C_{\ref{eq:c_d_min_d_max}} \cdot \frac{\log(1/\delta)}n
	\end{align*}
    with probability at least $1 - \delta$. The proof is finished.

\subsection{Proof of Theorem \ref{th:lower_bound}}
\label{sec:lower_bound_proof}

\begin{Th}[restatement of Theorem \ref{th:lower_bound}]
	Let $(X_1, \dots, X_n)$ be a sample of i.i.d. observations generated from a density 
	$\pstar$ satisfying \Cref{assu:AP}.
	Then for any estimate $\widehat\sfp$  of $\pstar$,  that is, measurable function of \(X_1,\ldots,X_n,\)
	it holds that
	\[
		\sup\limits_{\pstar } \E \JS(\widehat\sfp, \pstar)
		\gtrsim n^{-2\beta / (2\beta + d)}
	\]
	with a hidden constant depending on $d$ only.
\end{Th}

	Let
	\[
		h = h(n) = \left( \frac{1}{n \Lambda^{d} d^2} \right)^{1 / (2\beta + d)}.
	\]
	and let $M$ be the $(2h)$-packing number of $[1/3, 2/3]^d$. It is clear that $M \gtrsim h^{-d}$.
	Let $\{x_1, \dots, x_M\} \subseteq [1/3, 2/3]^d$ be a $2h$-separable set, that is, $\|x_i - 
	x_j\| > 2h$ for all $i \neq j$.
 
	Let $\varphi : \R^d \rightarrow \R$ be a function from the class $\H^{\beta + 2}(\R^d, H_{\varphi})$ with some $H_{\varphi} > 0$, such that $\text{supp}(\varphi) = \B(0, 1) \subset \R^d$, $\varphi$ attains its maximum at $0$, and $-\nabla^2 \varphi (0) \succeq I_{d \times d}$.
	Consider a parametric class of generators $\{ g_\theta : \theta \in \B(0,h) \subset \R^M\}$, where
	\begin{equation*}
	    \label{eq:g_theta}
		g_\theta^{-1}(x) = x + h^\beta \sum\limits_{j=1}^M \theta_j \nabla \varphi \left( \frac{x - 
		x_j}h \right),
		\quad x \in [0, 1]^d.
	\end{equation*}
	Here and further in the proof, $g^{-1}$ stands for the inverse map of $g$. For convenience, we split the rest of the proof into several steps.

    \medskip

    \noindent\textbf{Step 1: verifying the conditions.}
    \quad
    First, note that, according to the definition of $g_\theta$, we have
    \[
        \nabla g_\theta^{-1}(x) = I_{d \times d} + h^{\beta - 1} \sum\limits_{j = 1}^M \theta_j \nabla^2 \varphi\left( \frac{x - x_j}h \right).
    \]
    Since $\varphi$ is supported on the unit ball $\B(0, 1)$, the last expression simplifies to
    \[
        \nabla g_\theta^{-1}(x)
        = \begin{cases}
            I_{d \times d} + h^{\beta - 1} \theta_j \nabla^2 \varphi\left( \frac{x - x_j}h \right), \quad \text{if $x \in \B(x_j , h)$ for some $j \in \{1, \dots, M\}$,}\\
            I_{d \times d}, \quad \text{otherwise.}
        \end{cases}
    \]
    The fact that $\varphi \in \H^{\beta + 2}(\R^d, H_{\varphi})$ immediately implies that
    \begin{equation}
        \label{eq:g_theta_derivative}
        e^{-1/d} I_{d \times d}
        \preceq \left(1 - h^\beta H_\varphi \right) I_{d \times d}
        \preceq \nabla g_\theta^{-1}(x)
        \preceq \left(1 + h^\beta H_\varphi \right) I_{d \times d}
        \preceq e^{1/d} I_{d \times d},
    \end{equation}
    provided that
    \[
        h^\beta \leq \frac{e^{1/d} - 1}{H_\varphi} \land \frac{1 - e^{-1/d}}{H_\varphi},
    \]
    that is, if the sample size $n$ is large enough.
    Hence, we checked that $g_\theta^{-1} \in \H_{e^{1/d}}^{\beta + 1}([0, 1]^d, H_0)$ for some $H_0 > 0$ and each $\theta \in \B(0, h) \subset \R^M$. This yields that, for all $\theta \in \B(0, h) \subset \R^M$, $g_\theta$ belongs to a class $\H_{e^{1/d}}^{\beta + 1}([0, 1]^d, H^\circ)$ with some $H^\circ > 0$.
	
    Besides, for each $\theta \in \B(0, h)$, $g_\theta$ differs from $g_0(x) \equiv x$ only on $[1/3-h, 2/3+h]^d$. If $n$ is sufficiently large, then $h$ is small and $g_\theta^{-1}([1/3-h, 2/3+h]^d) \subset [0, 1]^d$. Then $g_\theta^{-1}([0, 1]^d) \subseteq [0, 1]^d$. Similarly, we can show that $g_\theta([0, 1]^d) \subseteq [0, 1]^d$. Hence, for any $\theta \in \B(0, h)$, $g_\theta$ is indeed a bijection between $[0, 1]^d$ and $[0, 1]^d$, provided that $h$ is small enough.

    \medskip

    \noindent\textbf{Step 2: a minimax lower bound on the accuracy of parametric estimation.}
    The next auxiliary result provides a lower bound on the accuracy of estimation of the 
	parameter $\theta$.
	\begin{Lem}
		\label{lem:van_trees}
		Let $Y_1, \dots, Y_n$ be i.i.d. random elements on $\Yset$, $\text{Vol}(\Yset) = 1$, drawn from the 
		uniform distribution on $\Yset$ and assume that a learner observes a sample $(X_1, \dots, X_n)$ 
		where $X_i = g_\theta(Y_i), i \in \{1, \ldots, n\}$. Under the assumptions of \Cref{th:lower_bound}, if $n$ is large enough, for any estimate $\widehat \theta$ taking its 
		values in $\B(0, h) \subset \R^M$ it holds that
		\[
		\sup\limits_{\|\theta\| \leq h} \E_\theta \|\widehat\theta - \theta\|^2
		\gtrsim \frac{M h^2}{n h^{2\beta + d} d^2 + 1}.
		\]
	\end{Lem}
	Here we write $\E_\theta$ to emphasize that the expectation is taken with respect 
	to the 
	probability measure $\PP_{g_\theta}^{\otimes n}$ where $\PP_{g_\theta}$ is a distribution of random element $X = g_{\theta}(Y)$ with $Y$ uniformly distributed on $[0,1]^d$. The proof of Lemma~\ref{lem:van_trees} is based on the van Trees inequality and it is postponed to \Cref{sec:van_trees_proof} 
	below.

    \medskip

    \noindent
    \textbf{Step 3: a minimax lower bound on the accuracy of density estimation.}
    Denote a density of the measure $\PP_{g_\theta}$ by $\sfp_{g_\theta}$. Our next goal is to convert the result of Lemma \ref{lem:van_trees} to the 
	lower bound on
	\[
		\inf\limits_{\widehat\theta} \sup\limits_{\|\theta\| \leq h} \E_\theta \|\sfp_{g_{\widehat\theta}}  - \sfp_{g_\theta} \|_{L_2(\Xset)}^2
	\]
	where the infimum is taken with respect to measurable functions $\widehat\theta$ of the sample satisfying assumptions of \Cref{lem:van_trees}. Recall that, for any $\theta \in \B(0, h)$, the log-density of r.v. $X = g_\theta(Y)$ 
	is given by
	\begin{equation}
	\label{eq:p_g_theta_density}
		\log \sfp_{g_\theta}(x)
		= \log \det\left( I_{d \times d} + h^{\beta - 1} \sum\limits_{j=1}^M \theta_j 
		\nabla^2 \varphi \left(\frac{x - x_j}h \right) \right).
	\end{equation}
	Before obtaining a minimax lower bound on $\sup\limits_{\|\theta\| \leq h} \E_\theta 
	\|\sfp_{g_{\widehat\theta}} - \sfp_{g_\theta} \|_{L_2(\Xset)}^2$, we prove some properties of the map $\theta \mapsto \log \sfp_{g_\theta}$.
	The next lemma shows that, for any $j \in \{1, \dots, M\}$, the partial derivative
	\[
		\frac{\partial \log \sfp_{g_\theta}(x)}{\partial \theta_j}
	\]
	is bounded away from zero in a vicinity of $x_j$.
	
	\begin{Lem}
		\label{lem:log_density}
		Let
		\[
			r_0 = \frac{h}{2(H_{\varphi} \vee 1)}.
		\]
		Then, under the assumptions of Theorem \ref{th:lower_bound}, for a sufficiently small $h > 0$ and any estimate $\widehat \theta$ taking its values in $\B(0, h) \subset \R^M$, it holds that
		\[
			\left| \frac{\partial \log \sfp_{g_\theta}(x)}{\partial \theta_j} \right|
			\geq h^{\beta - 1} d / 2, \quad \text{for all } x \in \B(x_j, r_0).
		\]
	\end{Lem}
	
	The proof of \Cref{lem:log_density} is postponed until \Cref{sec:log_density_proof}. \Cref{lem:log_density} implies that, for any $j \in \{1, \dots, M\}$,
	\[
		\left| \frac{\partial \sfp_{g_\theta}(x)}{\partial \theta_j} \right|
		= \frac{1}{\sfp_{g_\theta}(x)} \left| \frac{\partial \log \sfp_{g_\theta}(x)}{\partial 
		\theta_j} 
		\right|
		\geq \frac{h^{\beta - 1} d}{2 \sfp_{g_\theta}(x)}, \quad \text{for all } x \in \B(x_j, r_0).
	\]
	The equality \eqref{eq:p_g_theta_density} implies that for $x \in \B(x_j, r_0)$, $\frac{\partial \sfp_{g_\theta}(x)}{\partial \theta_i} = 0$ for $i \neq j$. Hence, for $\theta,\theta' \in \B(0,h)$, it holds that
	\begin{align*}
		\|\sfp_{g_{\theta'}}  - \sfp_{g_\theta} \|_{L_2(\Xset)}^2
		&
		\geq \sum\limits_{j=1}^M \int\limits_{\B(x_j, r_0)} \left( \sfp_{g_{\theta'}}(x) - 
		\sfp_{g_\theta}(x) \right)^2 dx
		\\&
		\geq \sum\limits_{j=1}^M \int\limits_{\B(x_j, r_0)} \frac{h^{2\beta - 2} 
		d^2(\theta'_j - \theta_j)^2 }{4 \sfp_{g_{\vartheta(x)}}(x)^2} dx
	\end{align*}
	for some $\vartheta(x) = t(x) \theta + (1 - t(x)) \theta'$, $t(x) \in [0, 1]$. The inequality \eqref{eq:g_theta_derivative} and \Cref{lem:pmin_pmax} imply that
	\begin{align*}
		\|\sfp_{g_{\theta'}}  - \sfp_{g_\theta} \|_{L_2(\Xset)}^2
		&
		\geq \sum\limits_{j=1}^M \int\limits_{\B(x_j, r_0)} \frac{h^{2\beta - 2}
		d^2}{4 e^2} (\theta'_{j} - \theta_j)^2 dx
		\\&
		\gtrsim \sum\limits_{j=1}^M \left(\frac{r_0}h\right)^d h^{2\beta - 2 + d} \; (\theta'_{j} - \theta_j)^2
		\\&
		\gtrsim h^{2\beta - 2 + d} \; \|\theta' - \theta\|^2,
	\end{align*}
	where we write $\gtrsim$ for inequality up to an absolute constant in power $d$. Hence, \Cref{lem:van_trees} implies that
	\begin{align*}
		\sup\limits_{\|\theta\| \leq h} \E_\theta 
		\|\sfp_{g_{\widehat\theta}}  - \sfp_{g_\theta} \|_{L_2(\Xset)}^2
		&
		\gtrsim h^{2\beta - 2 + d} \; 
		\sup\limits_{\|\theta\| \leq h} \E_\theta \|\widehat\theta - \theta\|^2
		\\&
		\gtrsim \frac{h^{2 \beta}}{n h^{2\beta + d} d^2 + 1}
	\end{align*}
	for any estimate $\widehat \theta$ taking its values in $\B(0, h)$.
	Since $h = \left(n d^2\right)^{- 1/(2\beta + d)}$, we get
	\begin{equation}
		\label{eq:parametric_lower_bound}
		\sup\limits_{\|\theta\| \leq h} \E_\theta 
		\|\sfp_{g_{\widehat\theta}}  - \sfp_{g_\theta} \|_{L_2(\Xset)}^2
		\gtrsim n^{-2\beta / (2\beta + d)},
	\end{equation}
    where the hidden constant depends on $d$ only.
	
    \medskip

    \noindent
    \textbf{Step 4: a minimax lower bound on the accuracy of nonparametric estimation.}\quad
    We now prove that \eqref{eq:parametric_lower_bound} yields
	\[
		\inf\limits_{\widehat\sfp} \sup\limits_{\pstar} \E \|\widehat\sfp - \pstar\|_{L_2(\Xset)}^2
		\gtrsim n^{-2\beta / (2\beta + d)},
	\]
	where $\pstar$ satisfies \Cref{assu:AP} and $\widehat\sfp$ is any estimate of $\pstar$.
	Let $C_{d}$ be the hidden constant in \eqref{eq:parametric_lower_bound}, that is,
	\begin{equation}
		\label{eq:parametric_lower_bound_expl}
		\sup\limits_{\|\theta\| \leq h} \E_\theta 
		\|\sfp_{g_{\widehat\theta}}  - \sfp_{g_\theta} \|_{L_2(\Xset)}^2
		\geq C_{d} \, n^{-2\beta / (2\beta + d)}.
	\end{equation}
	Let $\widehat\sfp$ be an arbitrary estimate. Since $\sup\limits_{\pstar} \E \|\widehat\sfp - \pstar\|_{L_2(\Xset)}^2 \geq \sup\limits_{\|\theta\| \leq h} \E_\theta 
	\|\widehat\sfp  - \sfp_{g_\theta} \|_{L_2(\Xset)}^2$, it is enough to show that
	\[
		\sup\limits_{\|\theta\| \leq h} \E_\theta \|\widehat\sfp  - \sfp_{g_\theta} 
		\|_{L_2(\Xset)}^2
		\gtrsim n^{-2\beta / (2\beta + d)}.
	\]
	Let us introduce
	\[
		\widetilde\theta \in \argmin\limits_{\|\theta_0\| \leq h} \|\widehat\sfp  - 
		\sfp_{g_{\theta_0}} \|_{L_2(\Xset)}^2.
	\]
	If
	\[
		\sup\limits_{\|\theta\| \leq h} \E_\theta \|\widehat\sfp  - \sfp_{g_{\widetilde \theta}} \|_{L_2(\Xset)}^2
		\geq \frac{C_{d}}{4} \, n^{-2\beta / (2\beta + d)},
	\]
	then, by the definition of $\widetilde \theta$,
	\begin{align*}
		\sup\limits_{\|\theta\| \leq h} \E_\theta \|\widehat\sfp  - \sfp_{g_{\theta}} 
		\|_{L_2(\Xset)}^2 
		&
		\geq \sup\limits_{\|\theta\| \leq h} \E_\theta \|\widehat\sfp  - \sfp_{g_{\widetilde 
		\theta}} \|_{L_2(\Xset)}^2
		\\&
		\geq \frac{C_{d}}{4} \, n^{-2\beta / (2\beta + d)}.
	\end{align*}
	On the other hand, if
	\[
		\sup\limits_{\|\theta\| \leq h} \E_\theta \|\widehat\sfp  - \sfp_{g_{\widetilde \theta}} \|_{L_2(\Xset)}^2
		\leq \frac{C_{d}}{4} \, n^{-2\beta / (2\beta + d)},
	\]
	then the Cauchy-Schwarz inequality and \eqref{eq:parametric_lower_bound_expl} 
	yield
	\begin{align*}
		\sup\limits_{\|\theta\| \leq h} \E_\theta \|\widehat\sfp  - \sfp_{g_{\theta}} 
		\|_{L_2(\Xset)}^2 
		&
		\geq \frac12 \sup\limits_{\|\theta\| \leq h} \E_\theta \|\sfp_{g_\theta}  - 
		\sfp_{g_{\widetilde \theta}} \|_{L_2(\Xset)}^2 - \sup\limits_{\|\theta\| \leq h} 
		\E_\theta \|\widehat\sfp  - \sfp_{g_{\widetilde \theta}} \|_{L_2(\Xset)}^2
		\\&
		\geq \frac{C_{d}}{4} \, n^{-2\beta / (2\beta + d)}.
	\end{align*}
	Finally, \Cref{lem:js_reg}, \Cref{lem:pmin_pmax}, and \eqref{eq:g_theta_derivative} yield that, for any $\widehat\sfp$, it holds almost surely
	\[
		 \JS(\widehat\sfp, \pstar)
		 \geq \frac{\|\widehat\sfp - \pstar\|_{L_2(\Xset)}^2}{8 e}.
	\]
	Hence, for any $\widehat\sfp$, we obtain that
	\[
		\sup\limits_{\pstar} \E \JS(\widehat\sfp, \pstar)
		\geq 
		\sup\limits_{\pstar} \frac{\E \|\widehat\sfp - \pstar\|_{L_2(\Xset)}^2}{8 e}
		\gtrsim C_{d} \, n^{-2\beta / (2\beta + d)},
	\]
	where $\gtrsim$ stands for an inequality up to an absolute constant.

\section{Conclusion and future directions}

Despite the huge recent interest to theoretical properties of generative adversarial networks, the existing papers mostly focus on the generalization ability of GANs, missing the issues occurring in their practical use. For instance, one of the most challenging problems in the GAN training process is the mode collapse phenomenon \citep{salimans2016improved,che2016mode}, which appears for various loss functions, including Wasserstein GANs, ``vanilla'' GANs, or GANs with other divergence measures. Moreover, in the present literature on Wasserstein GANs, the estimation error is mostly governed by the rates of convergence of the empirical measure to the population distribution. This leaves a question, whether WGANs are able to produce a distribution estimate, which is strictly better that just the empirical distribution. The first step in this direction was made in the recent work \citep{vardanyan23}, where the authors imposed regularity conditions on generators. This agrees with the papers on vanilla GANs in a sense that successful generalization requires smoothness of generators and discriminators. Practitioners often use regularization (see, e.g., \cite{gulrajani17}) for this purpose.

\acks{The publication was supported by the grant for research centers in the field of AI provided by the Analytical Center for the Government of the Russian Federation (ACRF) in accordance with the agreement on the provision of subsidies (identifier of the agreement 000000D730321P5Q0002) and the agreement with HSE University \textnumero70-2021-00139.}

\bigskip

\bibliography{bibliography}

\appendix

\section{Properties of smooth maps of random elements with smooth densities}
\label{app:densities}

\subsection{Proof of \Cref{lem:pf-pg}}
\label{sec:pf-pg_proof}

Due to \cref{assu:G'}, $\nabla g_w(y)$ is non-degenerate for all $w \in \Wset$ and all $y \in [0, 1]^d$.
Thus, $\det \nabla g_w$ does not change its sign, and, without loss of generality, we can 
assume that $\det \nabla g_u(g_u^{-1}(x))$ and $\det \nabla g_v(g_v^{-1}(x))$ are positive for all $x \in [0, 1]^d$.
Let $z = g_u^{-1}(x) \in [0, 1]^d$.
Then $x = g_u(z)$ and 
\begin{align*}
	\left\|g_u^{-1}(x) - g_w^{-1}(x)\right\|
	&
	= \left\| g_u^{-1}( g_u(z) ) - g_v^{-1}( g_u(z) ) \right\|
	= \left\| g_v^{-1}( g_v(z) ) - g_v^{-1}( g_u(z) ) \right\|
	\\&
	\leq \Lambda \| g_v(z) - g_u(z) \|
	\leq \Lambda \sqrt{d} \|g_v - g_u\|_{L_{\infty}([0, 1]^d)}
\end{align*}
by the mean value theorem for vector valued functions.
Furthermore, we have
\begin{align*}
	\bigl|\sfp_{u}(x) - \sfp_{v}(x)\bigr|
	&
	= \left| \det[\nabla g_u(g_u^{-1}(x))]^{-1} - \det[\nabla g_v(g_v^{-1}(x))]^{-1} \right|
	\\&
	\leq \min\left\{ \det[\nabla g_u(g_u^{-1}(x))], \det[\nabla g_v(g_v^{-1}(x))] \right\}^{-2}
	\\&\quad
	\cdot \left|\det \nabla g_u(g_u^{-1}(x)) - \det \nabla g_v(g_v^{-1}(x)) \right|
	\\&
	\leq \Lambda^{2d}\left|\det\nabla g_u(g_u^{-1}(x)) - \det\nabla g_v(g_v^{-1}(x))\right|.
\end{align*}
Here we used the fact that, due to \Cref{assu:G'},
\begin{align*}
	\det [\nabla g_u(g_u^{-1}(x))]^{-1}
	&
	= \sqrt{\det \left( [\nabla g_u(g_u^{-1}(x))]^{-\top} [\nabla g_u(g_u^{-1}(x))]^{-1} \right)}
	\\&
	\leq \sqrt{\det(\Lambda^2 \Id_{d\times d})}
	= \Lambda^d,
\end{align*}
and, similarly, $\det [\nabla g_v(g_v^{-1}(x))]^{-1} \leq \Lambda^d$.
Next, since for any $d \times d$ matrices $A$ and $B$ it holds that
\begin{align*}
	\left| \det A - \det B \right|
	&
	\leq \|A - B\|_{\mathrm{F}} \frac{\|A\|_{\mathrm{F}}^d - 
	\|B\|_{\mathrm{F}}^d}{\|A\|_{\mathrm{F}} - \|B\|_{\mathrm{F}}}
	\\&
	\leq d \max\{\|A\|_{\mathrm{F}}^d, \|B\|_{\mathrm{F}}^d\} \|A - B\|_{\mathrm{F}}
\end{align*}
and
\begin{align*}
	&
	\|\nabla g_u(g_u^{-1}(x))\|_{\mathrm{F}}^2 = \text{Tr}\left( \nabla g_u(g_u^{-1}(x))^\top \nabla 
	g_u(g_u^{-1}(x)) \right)
	\leq \Lambda^2 d,
	\\&
	\|\nabla g_v(g_v^{-1}(x))\|_{\mathrm{F}}^2 = \text{Tr}\left( \nabla g_v(g_v^{-1}(x))^\top \nabla 
	g_v(g_v^{-1}(x)) \right)
	\leq \Lambda^2 d,
\end{align*}
we obtain that
	\begin{align*}
	\left|\det \nabla g_u(g_u^{-1}(x)) - \det \nabla g_v(g_v^{-1}(x)) \right|
	&
	\leq \Lambda^{d} d^{1 + d/2} \left\Vert \nabla g_u(g_u^{-1}(x)) - \nabla g_v(g_v^{-1}(x)) \right\Vert_{\mathrm{F}}
	\\&
	\leq \Lambda^{d} d^{1 + d/2} \left\Vert \nabla g_u(g_u^{-1}(x)) - \nabla g_v(g_u^{-1}(x))\right\Vert_{\mathrm{F}}
	\\&\quad
	+ \Lambda^{d} d^{1 + d/2} \left\Vert \nabla g_v(g_u^{-1}(x)) - \nabla g_v(g_v^{-1}(x))\right\Vert_{\mathrm{F}}
	\\&
	\leq \Lambda^{d} d^{2 + d/2} \|g_u - g_v\|_{\H^1(\Yset)}
	\\&\quad
	+ \Lambda^{d} d^{2 + d/2} H_\G \|g_u^{-1}(x) - g_v^{-1}(x)\|
	\\&
	\leq \Lambda^{d} d^{2 + d/2} (1 + H_\G \Lambda \sqrt d) \|g_u - g_v\|_{\H^1(\Yset)}.
\end{align*}
Hence,
\begin{align*}
	\|\sfp_u - \sfp_v\|_{L_\infty(\Xset)}
	\leq d^{2 + d/2} \Lambda^{3d} (1 + H_\G \Lambda \sqrt d) \|g_u - g_v\|_{\H^1(\Yset)}.
\end{align*}

\subsection{Proof of \Cref{lem:pmin_pmax}}
\label{sec:pmin_pmax_proof}

According to the definition of $\H_\Lambda^2([0, 1]^d, H_\G)$, we have $\Lambda^{-2} I_{d \times d} \preceq \nabla g(y)^\top \nabla g(y) \preceq \Lambda^2 I_{d \times d}$. Hence,
\[
	\Lambda^{-d}
	= \sqrt{\det(\Lambda^{-2} \Id_{d \times d})}
	\leq \sqrt{\det\left(\nabla g(g^{-1}(x))^\top \nabla g(g^{-1}(x)) \right)}
	\leq \sqrt{\det(\Lambda^{2} \Id_{d \times d})}
	= \Lambda^d.
\]
Then the equality
\[
	\sfp(x) = |\det \nabla g(g^{-1}(x))|^{-1}
\]
yields that
\[
	\Lambda^{-d} \leq \sfp(x) \leq \Lambda^d.
\]

\section{Some properties of feed-forward neural networks with ReQU activations}

Let $N \in \mathbb{N}, \A = (p_0, p_1, \dots, p_{N+1}) \in \mathbb{N}^{N+2}$, and $f(x): \R^{p_0} \rightarrow \R^{p_{N+1}}$ be a neural network
\begin{align*}
f(x) = W_N \circ \sigma_{v_{N}} \circ W_{L-1} \circ \sigma_{v_{L-1}} \circ \dots \circ \sigma_{v_1} \circ W_0 \circ x\,,
\end{align*}
where $W_i \in \R^{p_{i+1} \times p_i}$ are weight matrices. Similarly to \citep[Lemma 5]{schmidt-hieber2020}, we introduce the following notations. For $k \in \{1,\dots,N+1\}$, $i \leq k$, $k \leq j \leq N$, we define functions $\BBB_{k,i}(x): \R^{p_{i-1}} \rightarrow \R^{p_{k}}$ and $\AAA_{j,k}f(x): \R^{p_{k-1}} \rightarrow \R^{p_{j+1}}$ as follows
\begin{equation}
\label{eq:A_plus_A_minus_def}
\begin{split}
\BBB_{k,i}(x) &= \requ_{v_{k}} \circ W_{k-1} \circ \requ_{v_{k-1}} \dots \circ \requ_{v_i} \circ W_{i-1} \circ x; \\
\AAA_{j,k}(x) &= W_{j}  \circ \requ_{v_{j}} \dots \circ W_{k} \circ \requ_{v_{k}} \circ W_{k-1} \circ x\,.
\end{split}
\end{equation}
Set by convention $\AAA_{N,N+2}(x) = \BBB_{0,1}f(x) = x$. For notation simplicity, we write $\AAA_{j}$ instead of $\AAA_{N,j}$, and $\BBB_{j}$ instead of $\BBB_{j,1}$. Note that with this notation $f(x) = \AAA_{1}(x)$.
\par
Let us introduce the functions 
\begin{equation}
\label{eq:f_one_f_two_def}
\begin{split}
\fone(x) &= \Wone_N \circ \requ_{\vone_{N}} \circ \Wone_{N-1} \circ \requ_{\vone_{N-1}} \circ \dots \circ \requ_{\vone_1} \circ \Wone_0 \circ x\,, \\
\ftwo(x) &= \Wtwo_N \circ \requ_{\vtwo_{N}} \circ \Wtwo_{N-1} \circ \requ_{\vtwo_{N-1}} \circ \dots \circ \requ_{\vtwo_1} \circ \Wtwo_0 \circ x\,,
\end{split}
\end{equation}
where the parameters $\Wone_{i}, \Wtwo_{i}, \vone_{i}, \vtwo_{i}$ satisfy
\[
\norm{\Wone_{i} - \Wtwo_{i}}_{\infty} \leq \eps, \quad \norm{\vone_{i} - \vtwo_{i}}_{\infty} \leq \eps\,, \quad \text{for all } i \in \{0,\dots,N\}\,.
\]

\subsection{Proof of \Cref{Lem:inf_norm_bound}}
\label{sec:inf_norm_bound}

Before we prove \Cref{Lem:inf_norm_bound}, we need the following auxiliary result.

\begin{Lem}
\label{lem:infty_norm_a_plus}
Let $x \in \R^{d}$, $\norm{x}_{\infty} \leq \Kxnorm$. Then for $k, i \in \{1,\dots,N\}, \, k \geq i$ 
\begin{equation}
\label{eq:A_plus_infty_bound}
\norm{\BBB_{k,i}(x)}_{\infty} \leq \biggl\{\prod_{\ell=1}^{k-i+1}(p_{k-\ell}+1)^{2^{\ell}}\biggr\} (\Kxnorm \vee 1)^{2^{k-i+1}}\,,
\end{equation}
where $\BBB_{k,i}(x)$ are defined in \eqref{eq:A_plus_A_minus_def}. Moreover, function $\AAA_{j,k}(x)$ is Lipshitz for $x,y \in \R^d: \norm{x}_{\infty} \leq \Kxnorm, \norm{y}_{\infty} \leq \Kxnorm$, that is,
\begin{equation}
\label{eq:A_minus_lip_bound}
\norm{\AAA_{j,k}(x) - \AAA_{j,k}(y)}_{\infty} \leq 2^{j-k+1} \prod_{\ell=0}^{j-k+1}(p_{j-\ell}+1)^{2^{\ell}}(\Kxnorm \vee 1)^{2^{j-k+1}}\norm{x-y}_{\infty}\,. 
\end{equation}
\end{Lem}

\begin{proof}{\bf of \Cref{lem:infty_norm_a_plus}}
The inequality \eqref{eq:A_plus_infty_bound} follows from an easy induction in $k$. Indeed, if $k = i$, $\BBB_{i,i} = \requ_{v^{(i)}} \circ W_{i-1}x$, and 
\[
\norm{\BBB_{i,i}}_{\infty} \leq (\Kxnorm p_{i-1}+1)^2 \leq (p_{i-1}+1)^2(\Kxnorm \vee 1)^{2}\,.
\]
Using $\norm{\BBB_{k,i}}_{\infty} \leq (\norm{\BBB_{k-1,i}}_{\infty} p_{k-1}+1)^2$ completes the proof. 
\par 
To prove \eqref{eq:A_minus_lip_bound}, we use an induction in $j$. Assume that \eqref{eq:A_minus_lip_bound} holds for any $k \in \{1,\dots,N+1\}$ and $j - 1 \geq k$. Then
\begin{align*}
\norm{\AAA_{j,k}(x) - \AAA_{j,k}(y)}_{\infty} 
&\leq p_j \norm{\BBB_{j,k}(x) - \BBB_{j,k}(y)}_{\infty} \\
&\leq 2p_{j}\norm{\AAA_{j-1,k}(x) - \AAA_{j-1,k}(y)}_{\infty} \bigl(\norm{\AAA_{j-1,k}(x)}_{\infty} \vee \norm{\AAA_{j-1,k}(y)}_{\infty}\bigr) \\
&\leq 2(p_{j}+1)(p_{j-1}+1) \prod_{\ell=1}^{j-k}(p_{j-\ell-1}+1)^{2^{\ell}}(\Kxnorm \vee 1)^{2^{j-k}} \times \\
&\qquad \norm{\AAA_{j-1,k}(x) - \AAA_{j-1,k}(y)}_{\infty}\,,
\end{align*}
and the statement follows from the elementary bound 
\begin{align*}
\norm{\AAA_{k-1,k}(x) - \AAA_{k-1,k}(y)}_{\infty} \leq p_{k-1}\norm{x-y}_{\infty} \leq p_{k-1}\norm{x-y}_{\infty} (\Kxnorm \vee 1)\,.
\end{align*}
\end{proof}

\begin{proof}{\bf of \Cref{Lem:inf_norm_bound}}
Denote by $\AAA_{j}^{(1)}, \BBB_{j}^{(1)}, \AAA_{j}^{(2)}, \BBB_{j}^{(2)}$ the corresponding functions in \eqref{eq:A_plus_A_minus_def}. Following \citep[Lemma 5]{schmidt-hieber2020}, we write
\begin{align*}
\norm{\fone(x) - \ftwo(x)}_{\infty} 
\leq \sum_{k=1}^{N+1}\norm{\AAA^{(1)}_{k+1} \circ \requ_{\vone_{k}} \circ \Wone_{k-1} \circ \BBB_{k-1}^{(2)}(x) - \AAA^{(1)}_{k+1} \circ \requ_{\vtwo_{k}} \circ \Wtwo_{k-1} \circ \BBB_{k-1}^{(2)}(x)}_{\infty}\,.
\end{align*}
Due to \Cref{lem:infty_norm_a_plus}, functions $\AAA^{(1)}_{k+1} = \AAA^{(1)}_{L,k+1}$ are Lipshitz, and
\begin{align*}
\norm{\fone(x) - \ftwo(x)}_{\infty}
&\leq \sum_{k=1}^{N+1}2^{N-k}\prod_{\ell=0}^{N-k}\bigl(p_{N-\ell} + 1\bigr)^{2^{\ell}}\left(\norm{\BBB^{(2)}_{k}(x)}_{\infty} \vee 1\right)^{2^{N-k}} \times \\
&\quad \norm{\requ_{\vone_{k}} \circ \Wone_{k-1} \circ \BBB^{(2)}_{k-1}(x) - \requ_{\vtwo_{k}} \circ \Wtwo_{k-1} \circ \BBB^{(2)}_{k-1}(x)}_{\infty}
\end{align*}
Note that
\begin{align*}
&\norm{\requ_{\vone_{k}} \circ \Wone_{k-1} \circ \BBB^{(2)}_{k-1}(x) - \requ_{\vtwo_{k}} \circ \Wtwo_{k-1} \circ \BBB^{(2)}_{k-1}(x)}_{\infty} \leq \eps (p_{k-1}+1)\left(\norm{\BBB^{(2)}_{k-1}(x)}_{\infty} \vee 1\right) \times \\
&\qquad \biggl(\norm{\Wone_{k-1} \circ \BBB^{(2)}_{k-1}(x)}_{\infty} + 1 \vee \norm{\BBB^{(2)}_{k}(x)}_{\infty} + 1\biggr) \leq 2\eps (p_{k-1}+1)^2\left(\norm{\BBB^{(2)}_{k-1}(x)}_{\infty} \vee 1\right)^2\,.
\end{align*}
Combining the previous bounds and \eqref{eq:A_plus_infty_bound} yields
\begin{align*}
\norm{\fone(x) - \ftwo(x)}_{\infty} 
&\leq \eps \sum_{k=1}^{N+1}2^{N-k+1}\left\{\prod_{\ell=0}^{N-k}\bigl(p_{N-\ell} + 1\bigr)^{2^{\ell}}\right\} \bigl(p_{k-1} + 1\bigr)^2 \left(\norm{\BBB^{(2)}_{k}(x)}^2_{\infty} \vee 1\right)^{2^{N-k}} \\
&\leq \eps (N+1) 2^{N} \prod_{\ell=0}^{N}(p_{\ell}+1)^{2^{N}}\,,
\end{align*}
and the statement follows.

\end{proof}

\subsection{Proof of \Cref{Lem:inf_norm_bound_derivative}}
\label{sec:inf_norm_bound_derivative}

Note that for $f(x)$ defined in \eqref{eq:f_one_f_two_def}, it holds
\begin{equation}
\label{eq:derivative}
\nabla f(x) = 2^{N} W_{N} \prod_{\ell=1}^{N} \left\{\operatorname{diag}\left[\AAA_{N-\ell,1}(x) + v_{N-\ell+1} \vee 0\right] W_{N-\ell}\right\}\,.
\end{equation}
Let us define for $j \in \{0,\dots,N-1\}$ and $i \in \{1,2\}$ the quantities 
\begin{align*}
\Delta^{(i)}_{j} &= \AAA_{j,1}(x) + v^{(i)}_{j+1} \vee 0\,.
\end{align*}
Note that $\diagone_{j}, \diagtwo_{j} \in \R^{p_{j+1}}$. With the triangular inequality,
\begin{align*}
&\norm{\nabla \fone(x) - \nabla \ftwo(x)}_{\infty} 
\leq 2^{N} \norm{\bigl(\Wone_{N} - \Wtwo_{N}\bigr)\prod_{\ell=1}^{N}\left\{\operatorname{diag}\left[\diagone_{N-\ell}\right] \Wone_{N-\ell}\right\}}_{\infty} + 2^{N}(p_{N}+1) \times \\
&\sum_{k=1}^{N} \norm{\prod_{\ell=1}^{k-1}\left\{\operatorname{diag}\left[\diagone_{N-\ell}\right] \Wone_{N-\ell}\right\}
\left(\operatorname{diag}\left[\diagone_{k}\right] \Wone_{k} - \operatorname{diag}\left[\diagtwo_{k}\right] \Wtwo_{k}\right)
\prod_{\ell=k+1}^{N}\left\{\operatorname{diag}\left[\diagtwo_{N-\ell}\right] \Wtwo_{N-\ell}\right\}}_{\infty}
\end{align*}
Proceeding as in \Cref{Lem:inf_norm_bound}, we obtain 
\begin{align*}
\norm{\operatorname{diag}\left[\diagone_{k}\right] \Wone_{k} - \operatorname{diag}\left[\diagtwo_{k}\right] \Wtwo_{k}}_{\infty} 
&\leq \eps(k+1)2^{k}\prod_{\ell=0}^{k}(p_{\ell}+1)^{2^k} + (p_k+1)\norm{\diagone_{k}}_{\infty} \\
&\leq 2\eps (k+1)2^{k} \prod_{\ell=0}^{k}(p_{\ell}+1)^{2^k}\,,
\end{align*}
and, similarly,
\begin{align*}
\norm{\operatorname{diag}\left[\diagone_{N-\ell}\right] \Wone_{N-\ell}}_{\infty} \leq \prod_{j=1}^{N-\ell}(p_{N-j}+1)^{2^{j}}\,.
\end{align*}
Combining the previous bounds, we obtain
\begin{align*}
\norm{\nabla \fone(x) - \nabla \ftwo(x)}_{\infty} \leq \eps 2^{N}\prod_{\ell=0}^{N}(p_{\ell}+1)^{2^{N}+1} + \eps N(N+1)2^{N}\prod_{\ell=0}^{N}(p_{\ell}+1)^{2^{N+1}+1}\,,  
\end{align*}
and the statement follows.

\section{Proofs of the auxiliary results}

\subsection{Proof of \Cref{lem:Delta_D_lower_bound}}
\label{sec:Delta_D_lower_bound_proof}

	Let $D^*(x) = \pstar(x) / (\pstar(x) + \sfp_w(x))$.
	By the definition of $\JS(\sfp_w, \pstar)$ and $L(w, \theta)$, it holds that
	\begin{align*}
		&
		\JS(\sfp_w, \pstar) - \log 2 - L(w, \theta)
		= \frac12 \int \left[ \pstar(x) \log\left(\frac{D^*(x)}{D_\theta(x)}\right)
		+ \sfp_w \log\left(\frac{1 - D^*(x)}{1 - D_\theta(x)}\right) \right] \rmd \mu
		\\& \qquad
		= \int \left[ D^*(x) \log\left(\frac{D^*(x)}{D_\theta(x)}\right)
		+ (1 - D^*(x)) \log\left(\frac{1 - D^*(x)}{1 - D_\theta(x)}\right) \right] \frac{\pstar(x) + \sfp_w(x)}2 \, \rmd \mu.	
	\end{align*}
	Let $a \in [D_{\min}, D_{\max}]$ and introduce a function
	\[
		h_a(v) = (a + v) \log\left( 1 + \frac va \right) + (1 - a - v) \log\left( 1 - \frac v{1- a} \right).
	\]
	Then it holds that
	\begin{align*}
		\JS(\sfp_w, \pstar) - \log 2 - L(w, \theta)
		&
		= \int h_{D_\theta(x)}(D^*(x) - D_\theta(x)) \frac{\pstar(x) + \sfp_w(x)}2 \, \rmd \mu.
	\end{align*}
	Fix any $a \in [D_{\min}, D_{\max}]$ and prove that $h_a(v) \geq v^2$ for any $v \in [-a, 1 - a]$.
	For this purpose, compute the derivatives of $h_a(v)$:
	\begin{align*}
		&
		h_a'(v) = \log\left( 1 + \frac va \right) - \log\left(1 - \frac{v}{1 - a} \right),
		\\&
		h_a''(v) = \frac1{(a + v)(1 - a - v)} \geq 4
		\quad \text{for all $v \in (-a, 1 - a)$}.
	\end{align*}
	This yields that
	\[
		h_a(v) \geq h_a(0) + h'_a(0) v + 2 v^2 = 2 v^2
		\quad \text{for all $v \in [-a, 1 - a]$.}
	\]
	Hence,
	\begin{align*}
		\JS(\sfp_w, \pstar) - \log 2 - L(w, \theta)
		&
		= \int h_{D_\theta(x)}(D^*(x) - D_\theta(x)) \frac{\pstar(x) + \sfp_w(x)}2 \, \rmd \mu
		\\&
		\geq \int (D^*(x) - D_\theta(x))^2 (\pstar(x) + \sfp_w(x)) \, \rmd \mu.
		\\&
		= \left\| \frac{\pstar}{\pstar + \sfp_w} - D_\theta \right\|_{L_2(\pstar + \sfp_w)}^2.
	\end{align*}
	To prove \eqref{eq:delta_D_upper_bound}, it is enough to show that the following inequality holds for any $a \in [D_{\min}, D_{\max}]$ and any $v \in [-a, 1 - a]$:
    \[
        h_a(v) \leq \frac{C_a^2 v^2}{(C_a - 1)^2 a(1 - a)},
        \quad \text{where} \quad
        C_a = 1 + \sqrt{\frac{a}{(1 - a) \log(1 / (1 - a))} \wedge \frac{1 - a}{a \log(1 / a)}}.
    \]
    Then it is easy to observe that
    \[
        C_a \geq C_{\ref{eq:c_delta}} > 1
        \quad \text{and} \quad
        a(1 - a) \geq D_{\min}(1 - D_{\max})
        \quad \text{for all $a \in [D_{\min}, D_{\max}]$,}
    \]
    which yields
    \[
        h_a(v) \leq \frac{C_{\ref{eq:c_delta}}^2 v^2}{(C_{\ref{eq:c_delta}} - 1)^2 D_{\min}(1 - D_{\max})}
        \quad \text{for all $a \in [D_{\min}, D_{\max}]$ and $v \in [-a, 1 - a]$.}
    \]
    Hence,
    \begin{align*}
		&
        \JS(\sfp_w, \pstar) - \log 2 - L(w, \theta)
		= \int h_{D_\theta(x)}(D^*(x) - D_\theta(x)) \frac{\pstar(x) + \sfp_w(x)}2 \, \rmd \mu
		\\&
		\leq \frac{C_{\ref{eq:c_delta}}^2 }{(C_{\ref{eq:c_delta}} - 1)^2 D_{\min}(1 - D_{\max})} \int (D^*(x) - D_\theta(x))^2 (\pstar(x) + \sfp_w(x)) \, \rmd \mu
		\\&
		= \frac{C_{\ref{eq:c_delta}}^2 }{(C_{\ref{eq:c_delta}} - 1)^2 D_{\min}(1 - D_{\max})} \left\| \frac{\pstar}{\pstar + \sfp_w} - D_\theta \right\|_{L_2(\pstar + \sfp_w)}^2.
	\end{align*}
    Note that, for any $v \in [-a / C_a, (1 - a) / C_a]$, we have
    \[
        h_a''(v)
        = \frac1{(a + v)(1 - a - v)}
        \leq \frac1{(1 - 1/C_a)^2 a(1 - a)}.
    \]
    Since $h_a(0) = h_a'(0) = 0$, this yields that
    \[
        h_a(v) \leq \frac{C_a^2 v^2}{2(C_a - 1)^2 a(1 - a)}.
    \]
    Fix any $v \in [-a, -a / C_a]$ and consider $h_a(v)$. Since $h_a$ decreases on $[-a, 0]$, we obtain that
    \[
        h_a(v)
        \leq h_a(-a)
        = \log\frac1{1 - a}.
    \]
    On the other hand, if $v \in [-a, -a / C_a]$, then $v^2 \geq a^2 / C_a^2$. Taking into account that
    \[
        (C_a - 1)^2 \leq \frac{a}{(1 - a) \log(1 / (1 - a))}
    \]
    due to the definition of $C_a$, we obtain that, for any $v \in [-a, -a / C_a]$, the following inequality holds:
    \[
        \frac{C_a^2 v^2}{(C_a - 1)^2 a (1 - a)}
        \geq \frac{a}{(C_a - 1)^2 (1 - a)}
        \geq \log \frac1{1 - a}
        \geq h_a(v).
    \]
    Similarly, since $h_a$ increases on $[0, 1 - a]$, it holds that
    \[
        h_a(v)
        \leq h_a(1 - a)
        = \log\frac1{a}
        \quad \text{for all $v \in \left[\frac{1 - a}{C_a}, 1 - a \right]$.}
    \]
    At the same time, $v^2 \geq (1 - a)^2 / C_a^2$ for any $v \in [(1 - a) / C_a, 1 - a]$. Using the inequality
    \[
        (C_a - 1)^2 \leq \frac{1 - a}{a \log(1 / a)},
    \]
    which follows from the definition of $C_a$, we obtain that
    \[
        \frac{C_a^2 v^2}{(C_a - 1)^2 a(1 - a)}
        \geq \frac{1 - a}{a (C_a - 1)^2}
        \geq \log \frac1a
        \geq h_a(v),
        \quad \text{for all $v \in \left[\frac{1 - a}{C_a}, 1 - a \right]$,}
    \]
    and the proof is finished.

\subsection{Proof of \Cref{lem:One_Delta_Bernstein}}
\label{sec:One_Delta_Bernstein_proof}

\begin{proof}{\bf of \Cref{lem:One_Delta_Bernstein}}\quad
	The proof of the lemma is quite long, so we split it in several steps.
	
	\medskip
	
	\noindent{\bf Step 1: Bernstein's bound.}\quad
	Let us recall that
	\[
		L_n(w, \theta) = \frac1{2n} \sum\limits_{i=1}^n \log D_{\theta}(X_i) + \frac1{2n} \sum\limits_{i=1}^n \log (1 - D_{\theta}(g_w(Y_i))),	
	\]
	and $L(w, \theta)$ is the expectation of $L_n(w, \theta)$. Assumption \ref{assu:D} yields that the absolute values of the random variables $\log D_{\theta}(X_i)$ and $\log (1 - D_{\theta}(g_w(Y_i))$ do not exceed $3 C_{\ref{eq:C_D_def}} / 2$, where the constant $C_{\ref{eq:C_D_def}}$ is given by \eqref{eq:C_D_def}. Then the Bernstein inequality implies that, for any $\delta \in (0, 1)$, with probability at least $1 - \delta$, we have
	\begin{equation}
		\label{eq:bernstein_1}
		\left| L_n(w, \theta) - L(w, \theta) \right|
		\leq \sqrt{ 2 \Var[L_n(w, \theta)] \log(2/\delta)} + \frac{2 C_D \log(2/\delta)}{3 n}.
	\end{equation}
	It remains to prove that
	\[
		4n \Var[L_n(w, \theta)]
		\leq C_{\ref{eq:c_D_var}} (9 \JS(\sfp_w, \pstar) + \Delta(w, \theta)).
	\]
	\medskip
	\noindent{\bf Step 2a: bounding the variance.}\quad
	The variance of $L_n(w, \theta_w^*)$ satisfies the inequality
	\begin{align*}
		4n \Var[L_n(w, \theta)]
		&
		= \Var\big[\log D_{\theta}(X_i)\big] + \Var\big[\log (1 - D_{\theta}(g_w(Y_i)))\big]
		\\&
		= \Var\big[\log(2 D_{\theta}(X_i))\big] + \Var\big[\log (2 - 2 D_{\theta}(g_w(Y_i)))\big]
		\\&
		\leq  \E \log^2 (2 D_{\theta}(X_1)) + \log^2 (2 - 2 D_{\theta}(g_w(Y_1)))
		\\&
		= \int \left[ \pstar(x) \log^2 (2 D_{\theta}(x))  + \sfp_w(x) \log^2 (2 - 2 D_{\theta}(x)) \right] \, \rmd \mu.
	\end{align*}
	Note that
	\begin{align}
		\label{eq:var_bound}
		 &\notag
		 \int \left[ \pstar(x) \log^2 (2 D_{\theta}(x))  + \sfp_w(x) \log^2 (2 - 2 D_{\theta}(x)) \right] \rmd \mu
		 \\&\notag
		 = \int \frac{\pstar(x)}{\pstar(x) + \sfp_w(x)} \log^2(2 D_{\theta}(x)) \, (\pstar(x) + \sfp_w(x)) \, \rmd \mu
		 \\&\notag\quad
		 + \frac{\sfp_w(x)}{\pstar(x) + \sfp_w(x)} \log^2 (2 - 2 D_{\theta}(x)) \, (\pstar(x) + \sfp_w(x)) \, \rmd \mu
		 \\&
		 = \int \left[ \frac{\pstar(x)}{\pstar(x) + \sfp_w(x)} - \frac12 \right] \log^2 (2 D_{\theta}(x))
		 (\pstar(x) + \sfp_w(x)) \, \rmd \mu
		 \\&\notag\quad
		 + \int \left[ \frac{\sfp_w(x)}{\pstar(x) + \sfp_w(x)} - \frac12 \right] \log^2 (2 - 2 D_{\theta}(x))
		 (\pstar(x) + \sfp_w(x)) \, \rmd \mu
		 \\&\notag\quad
		 + \int \left[ \log^2 (2 D_{\theta}(x)) + \log^2 (2 - 2 D_{\theta}(x)) \right] \frac{\pstar(x) + \sfp_w(x)}2 \, \rmd \mu.
	\end{align}
	
	\medskip
	\noindent{\bf Step 2b: bounding the variance, the first term.}\quad
	Consider the first term in the right-hand side. The Cauchy-Schwarz inequality yields that
	\begin{align*}
		&
		\int \left[ \frac{\pstar(x)}{\pstar(x) + \sfp_w(x)} - \frac12 \right] \log^2 (2 D_{\theta}(x))
		(\pstar(x) + \sfp_w(x)) \, \rmd \mu
		\\&
		\leq \left\| \frac{\pstar}{\pstar + \sfp_w} - \frac12 \right\|_{L_2(\pstar + \sfp_w)}
		\left\| \log^2 (2 D_{\theta}) \right\|_{L_2(\pstar + \sfp_w)}. 
	\end{align*}
	Applying \Cref{lem:js_reg}, we obtain that
	\begin{align*}
		&
		\int \left[ \frac{\pstar(x)}{\pstar(x) + \sfp_w(x)} - \frac12 \right] \log^2 (2 D_{\theta}(x))
		(\pstar(x) + \sfp_w(x)) \, \rmd \mu
		\\&
		\leq 2 \sqrt{\JS(\sfp_w, \pstar)} \left\| \log^2 (2 D_{\theta}) \right\|_{L_2(\pstar + \sfp_w)}.
	\end{align*}
	According to Assumption \ref{assu:D}, $D_{\theta}(x) \in [D_{\min}, D_{\max}] \subset (0, 1)$. The map $g(u) = \log^2(2u)$ is convex on $[D_{\min}, 1]$. Thus, it holds that
	\[
		g(u) \leq g(1/2) + \frac{g(1) (u - 1/2)}{1 - 1/2}
		\quad \text{for all $u \in [1/2, D_{\max}]$}
	\]
	and
	\[
		g(u) \leq g(1/2) + \frac{g(D_{\min}) (u - 1/2)}{D_{\min} - 1/2}
		\quad \text{for all $u \in [D_{\min}, 1/2]$.}
	\]
	This implies that
	\[
		\left\| \log^2 (2 D_{\theta}) \right\|_{L_2(\pstar + \sfp_w)}
		\leq \left( \frac{\log^2(2 D_{\min})}{1/2 - D_{\min}} \vee 2 \log^2 2 \right) \left\| D_{\theta} - \frac12 \right\|_{L_2(\pstar + \sfp_w)}.
	\]
	Hence, we finally obtain that
	\begin{align*}
		&
		\int \left[ \frac{\pstar(x)}{\pstar(x) + \sfp_w(x)} - \frac12 \right] \log^2 (2 D_{\theta}(x))
		(\pstar(x) + \sfp_w(x)) \, \rmd \mu
		\\&
		\leq 2 \left( \frac{\log^2(2 D_{\min})}{1/2 - D_{\min}} \vee 2 \log^2 2 \right) \sqrt{\JS(\sfp_w, \pstar)} \left\| D_{\theta} - \frac12 \right\|_{L_2(\pstar + \sfp_w)}.
	\end{align*}
	
	\medskip
	\noindent{\bf Step 2c: bounding the variance, the second term.}\quad
	For the second term in the right-hand side of \eqref{eq:var_bound} we similarly have
	\begin{align*}
		&
		\int \left[ \frac{\sfp_w(x)}{\pstar(x) + \sfp_w(x)} - \frac12 \right] \log^2 (2 - 2 D_{\theta_w}(x))
		(\pstar(x) + \sfp_w(x)) \, \rmd \mu
		\\&
		\leq 2 \left( \frac{\log^2(2 - 2 D_{\max})}{D_{\max} - 1/2} \vee 2 \log^2 2 \right) \sqrt{\JS(\sfp_w, \pstar)} \left\| D_{\theta} - \frac12 \right\|_{L_2(\pstar + \sfp_w)}.
	\end{align*}
	
	\medskip
	\noindent{\bf Step 2d: bounding the variance, the third term.}\quad
	It remains to bound
	\[
		\int \left[ \log^2 (2 D_{\theta}(x)) + \log^2 (2 - 2 D_{\theta}(x)) \right]
		\frac{\pstar(x) + \sfp_w(x)}2 \, \rmd \mu.
	\]
	Consider the function $h(u) = \log^2(2u) + \log^2(2 - 2u)$, $u \in [D_{\min}, D_{\max}]$. It is easy to observe that $h(1/2) = 0$, $h'(1/2) = 0$, and
	\[
		h''(u)
		= \frac{2 \log(e/(2u))}{u^2} + \frac{2 \log(e/(2 - 2u))}{(1 - u)^2}
		\leq \frac{2 \log(e/(2 D_{\min}))}{D_{\min}^2} + \frac{2 \log(e/(2 - 2D_{\max}))}{(1 - D_{\max})^2}.
	\]
	Hence, it holds that
	\[
		h(u)
		\leq \left( \frac{\log(e/(2 D_{\min}))}{D_{\min}^2}
		+ \frac{\log(e/(2 - 2D_{\max}))}{(1 - D_{\max})^2} \right) \left(u - 1/2 \right)^2
		\quad \text{for all $u \in [D_{\min}, D_{\max}]$,}
	\]
	and, therefore,
	\begin{align*}
		&
		\int \left[ \log^2 (2 D_{\theta}(x)) + \log^2 (2 - 2 D_{\theta}(x)) \right]
		\frac{\pstar(x) + \sfp_w(x)}2 \, \rmd \mu
		\\&
		\leq \left( \frac{\log(e/(2 D_{\min}))}{2 D_{\min}^2}
		+ \frac{\log(e/(2 - 2D_{\max}))}{2(1 - D_{\max})^2} \right)
		\left\| D_{\theta} - \frac12 \right\|_{L_2(\pstar + \sfp_w)}^2.
	\end{align*}
	
	\medskip
	\noindent{\bf Step 3: final part.}\quad
	Steps 2a--d yield that
	\begin{align*}
		4n \Var[L_n(w, \theta)]
		&
		\leq 2 \left( \frac{\log^2(2 D_{\min})}{1/2 - D_{\min}} \vee 2 \log^2 2 \right) \sqrt{\JS(\sfp_w, \pstar)} \left\| D_{\theta} - \frac12 \right\|_{L_2(\pstar + \sfp_w)}
		\\&\quad
		+ 2 \left( \frac{\log^2(2 - 2 D_{\max})}{D_{\max} - 1/2} \vee 2 \log^2 2 \right) \sqrt{\JS(\sfp_w, \pstar)} \left\| D_{\theta_w} - \frac12 \right\|_{L_2(\pstar + \sfp_w)}
		\\&\quad
		+ \left( \frac{\log(e/(2 D_{\min}))}{2 D_{\min}^2}
		+ \frac{\log(e/(2 - 2D_{\max}))}{2(1 - D_{\max})^2} \right)
		\left\| D_{\theta} - \frac12 \right\|_{L_2(\pstar + \sfp_w)}^2.
	\end{align*}
	By the Cauchy-Schwarz inequality,
	\[
		4n \Var[L_n(w, \theta)]
		\leq C_{\ref{eq:c_D_var}} \JS(\sfp_w, \pstar) + C_{\ref{eq:c_D_var}}  \left\| D_{\theta} - \frac12 \right\|_{L_2(\pstar + \sfp_w)}^2,
	\]
	where $C_{\ref{eq:c_D_var}}$ is given by \eqref{eq:c_D_var}.
	Consider the norm $\| D_{\theta_w^*} - 1/2 \|_{L_2(\pstar + \sfp_w)}$. By the triangle inequality, it holds that
	\[
		\left\| D_{\theta_w^*} - \frac12 \right\|_{L_2(\pstar + \sfp_w)}
		\leq \left\| D_{\theta_w^*} - \frac{\pstar}{\pstar + \sfp_w} \right\|_{L_2(\pstar + \sfp_w)}
		+ \left\| \frac{\pstar}{\pstar + \sfp_w} - \frac12 \right\|_{L_2(\pstar + \sfp_w)}.
	\]
	Applying \Cref{lem:Delta_D_lower_bound} and \Cref{lem:js_reg}, we obtain that
	\begin{align*}
		\left\| D_{\theta} - \frac12 \right\|_{L_2(\pstar + \sfp_w)}^2
		&
		\leq \left( \sqrt{\JS(\sfp_w, \pstar) - \log 2 - L(w, \theta)} + 2 \sqrt{\JS(\sfp_w, \pstar)} \right)^2
		\\&
		\leq \left( \sqrt{\Delta(w, \theta)} + 2 \sqrt{\JS(\sfp_w, \pstar)} \right)^2
		\\&
		\leq \Delta(w, \theta) + 8 \JS(\sfp_w, \pstar).
	\end{align*}
	Then it holds that
	\[
		4n \Var[L_n(w, \theta)]
		\leq 9 C_{\ref{eq:c_D_var}} \JS(\sfp_w, \pstar) + C_{\ref{eq:c_D_var}} \Delta(w, \theta),
	\]
	and the claim of the lemma follows.
\end{proof}

\subsection{Proof of \Cref{lem:Delta-Lip}}
\label{sec:Delta-Lip_proof}
	Since the statement of the lemma contains four inequalities, we split the proof in several steps for the sake of clarity.
	
	\medskip
	\noindent
	{\bf Step 1.}\quad
	Prove that, for any $w_1, w_2 \in \Wset$ and $\theta \in \Tset$, we have
	\[
		\left| L_n(w_1, \theta) - L_n(w_2, \theta) \right|
		\leq \frac{\sfl_\G \sfl_\Xset \|w_1 - w_2\|_\infty}{2 - 2D_{\max}} \quad \text{almost surely.}
	\]
	It holds that
	\[
		L_n(w_1, \theta) - L_n(w_2, \theta)
		= \frac1{2n} \sum\limits_{i=1}^n \log \frac{1 - D_\theta(g_{w_1}(Y_i))}{1 - D_\theta(g_{w_2}(Y_i))}.
	\]
	Since the map $x \mapsto \log (1 - x)$ is Lipschitz on $[D_{\min}, D_{\max}]$ with the constant $1/ (1 - D_{\max})$, Assumption \ref{assu:D} yields that
	\begin{align*}
		\left| L_n(w_1, \theta) - L_n(w_2, \theta) \right|
		&
		\leq \frac1{2n} \sum\limits_{i=1}^n \frac{\left| D_\theta(g_{w_1}(Y_i)) - D_\theta(g_{w_2}(Y_i)) \right|}{1 - D_{\max}}
		\\&
		\leq \frac1{2n} \sum\limits_{i=1}^n \frac{\sfl_\Xset \left\| g_{w_1}(Y_i) - g_{w_2}(Y_i) \right\|}{1 - D_{\max}}.
	\end{align*}
	Due to Assumption \ref{assu:G}, we have
	\[
		\left| L_n(w_1, \theta) - L_n(w_2, \theta) \right|
		\leq \frac{\sfl_\G \sfl_\Xset \|w_1 - w_2\|_\infty}{2 - 2D_{\max}} \quad \text{almost surely.}
	\]
	
	\medskip
	\noindent
	{\bf Step 2.}\quad
	Let us show that, for any $w \in \Wset$ and $\theta_1, \theta_2 \in \Tset$, we have
	\[
		\left| L_n(w, \theta_1) - L_n(w, \theta_2) \right|
		\leq \frac{L_\Tset \| \theta_1 - \theta_2 \|_\infty}{D_{\min} \wedge (1 - D_{\max})} \quad \text{almost surely.}
	\]
	It holds that
	\[
		L_n(w, \theta_1) - L_n(w, \theta_2)
		= \frac1{2n} \sum\limits_{i = 1}^n \log \frac{D_{\theta_1}(X_i)}{D_{\theta_2}(X_i)}
		+ \frac1{2n} \sum\limits_{i = 1}^n \log \frac{1 - D_{\theta_1}(g_w(Y_i))}{1 - D_{\theta_2}(g_w(Y_i))}
	\]
	The maps $x \mapsto \log x$ and $x \mapsto \log (1 - x)$ are Lipschitz on $[D_{\min}, D_{\max}]$ with the constants $1 / D_{\min}$ and $1 / (1 - D_{\max})$, respectively. Hence, we have
	\[
		\left| L_n(w, \theta_1) - L_n(w, \theta_2) \right|
		\leq \frac1{2n} \sum\limits_{i = 1}^n \frac{| D_{\theta_1}(X_i) - D_{\theta_2}(X_i) |}{D_{\min}}
		+ \frac1{2n} \sum\limits_{i = 1}^n \frac{| D_{\theta_1}(X_i) - D_{\theta_2}(X_i) |}{1 - D_{\max}}.
	\]
	Then, due to Assumption \ref{assu:D}, it holds that
	\[
		\left| L_n(w, \theta_1) - L_n(w, \theta_2) \right|
		\leq \frac{\sfl_\Tset \| \theta_1 - \theta_2 \|_\infty}{2 D_{\min}}
		+ \frac{\sfl_\Tset \| \theta_1 - \theta_2 \|_\infty}{2 - 2 D_{\max}}
		\leq \frac{\sfl_\Tset \| \theta_1 - \theta_2 \|_\infty}{D_{\min} \wedge (1 - D_{\max})}
	\]
	almost surely.
	
	\medskip
	\noindent
	{\bf Step 3.}\quad
	Finally, due to the Jensen inequality, it holds that
	\begin{align*}
		\left| L(w_1, \theta) - L(w_2, \theta) \right|
		&
		= \left| \E L_n(w_1, \theta) - \E L_n(w_2, \theta) \right|
		\\&
		\leq \E \left| L_n(w_1, \theta) - L_n(w_2, \theta) \right|
		\\&
		\leq \frac{\sfl_\G \sfl_\Xset \|w_1 - w_2\|_\infty}{2 - 2D_{\max}}
	\end{align*}
	and, similarly,
	\begin{align*}
		\left| L(w, \theta_1) - L(w, \theta_2) \right|
		&
		= \left| \E L_n(w, \theta_1) - \E L_n(w, \theta_2) \right|
		\\&
		\leq \E \left| L_n(w, \theta_1) - L_n(w, \theta_2) \right|
		\\&
		\leq \frac{\sfl_\Tset \| \theta_1 - \theta_2 \|_\infty}{D_{\min} \wedge (1 - D_{\max})}.
	\end{align*}

\subsection{Proof of \Cref{lem:js_reg}}
\label{sec:js_reg_proof}

	Let $h(u) = (1 + u) \log(1 + u) + (1 - u) \log(1 - u)$, $u \in [-1, 1]$, and rewrite the Jensen-Shannon divergence between $\sfp$ and $\sfq$ in the following form:
	\begin{align*}
		\JS(\sfp, \sfq)
		&
		= \frac12 \left[ \int \log\left( \frac{2 \sfp(x)}{\sfp(x) + \sfq(x)}\right) \sfp(x) \rmd \mu
		+ \int \log \frac{2 \sfq(x)}{\sfp(x) + \sfq(x)} \sfq(x) \rmd \mu \right]
		\\&
		= \int \Bigg[ \frac{2 \sfp(x)}{\sfp(x) + \sfq(x)} \log \frac{2 \sfp(x)}{\sfp(x) + \sfq(x)} 
		+  \frac{2 \sfq(x)}{\sfp(x) + \sfq(x)} \log \frac{2 \sfq(x)}{\sfp(x) + \sfq(x)} \Bigg] \frac{\sfp(x) + \sfq(x)}4 \rmd \mu 
		\\&
		= \int h\left( \frac{\sfp(x) - \sfq(x)}{\sfp(x) + \sfq(x)} \right) \frac{\sfp(x) + \sfq(x)}4 \rmd \mu.
	\end{align*}
	Note that the function $h(u) / u^2$ is even and increasing on $[0, 1]$. Hence, it attains its maximum on $[-1, 1]$ at the points $u = -1$ and $u = 1$ and its minimum at $u = 0$. This yields that
	\[
		1 = \lim\limits_{v \rightarrow 0} \frac{h(v)}{v^2} \leq \frac{h(u)}{u^2} \leq \frac{h(1)}{1^2} = 2\log 2
		\quad \text{for all $u \in [-1, 1]$.}
	\]
	Since $(\sfp(x) - \sfq(x)) / (\sfp(x) + \sfq(x)) \in [-1, 1]$ for all $x$ from $\text{supp}(\sfp) \cup \text{supp}(\sfq)$, it holds that
	\[
		\frac14 \int \frac{(\pstar(x) - \sfp_w(x))^2}{\pstar(x) + \sfp_w(x)} \rmd \mu
		\leq \JS(\sfp_w, \pstar)
		\leq \frac{\log 2}2 \int \frac{(\pstar(x) - \sfp_w(x))^2}{\pstar(x) + \sfp_w(x)} \rmd \mu.
	\]

\subsection{Proof of \Cref{co:js_metric_holder}}
\label{sec:js_metric_holder_proof}

Fix any $u$ and $v$ from $\Wset$.
The fact that the square root of the $\JS$-divergence is a metric (see, e.g., \citep{endres03}) implies that
\[
	\left| \sqrt{\JS(\sfp_u, \pstar)} - \sqrt{\JS(\sfp_v, \pstar)} \right|
	\leq \sqrt{\JS(\sfp_u, \sfp_v)}
\]
Then the claim of the corollary follows from the inequalities
\begin{align*}
	\JS(\sfp_u, \sfp_v)
	&
	\leq \frac{\log 2}2  \int \frac{(\sfp_u(x) - \sfp_v(x))^2}{\sfp_u(x) + \sfp_v(x)} \rmd \mu
	\\&
    \leq \frac{\log 2}2  \int \left|\sfp_u(x) - \sfp_v(x) \right| \rmd \mu 
	\\&
	\leq \frac{\log 2}2 \left\| \sfp_u - \sfp_v \right\|_{L_\infty(\Xset)}^{1/2} \int \left|\sfp_u(x) - \sfp_v(x) \right|^{1/2} \rmd \mu
    \\&
    \leq \frac{\log 2}2 \left\| \sfp_u - \sfp_v \right\|_{L_\infty(\Xset)}^{1/2} \sqrt{\int \left|\sfp_u(x) - \sfp_v(x) \right| \rmd \mu}
	\\&
    \leq \frac{\sfl_\sfp \log 2}{\sqrt 2} \|u - v\|_\infty^{1/2}
    < \sfl_\sfp \|u - v\|_\infty^{1/2}.
\end{align*}
where we used \Cref{lem:js_reg} and Assumption \ref{assu:p}.

\section{Proofs related to Theorem \ref{th:lower_bound}}
\subsection{Proof of Lemma \ref{lem:van_trees}}
\label{sec:van_trees_proof}

The main ingredient we use is the multivariate van Trees inequality (see \citep{vt68}, p. 
72, and \citep{gl95}).
Choose a density $\lambda$ of a prior distribution on the set of parameters $\{\theta \in 
\R^M : \|\theta\| \leq h \}$ of the form
\[
	\lambda(\theta) = \frac{1}{h^M} \lambda_0\left( \frac{\theta}{h} \right),
\]
where $\lambda_0$ is a smooth density supported on $\B(0, 1) \subset \R^M$.
We write $\E_\lambda$ for the expectation with respect to the density $\lambda$.
Then, for any estimate $\widehat\theta$, it holds that
\begin{equation}
	\label{van_trees}
	\sup\limits_{\|\theta\| \leq h} \E_\theta \|\widehat \theta - \theta\|^2
	\geq \E_{\lambda} \E_{\theta} \|\widehat\theta - \theta\|^2
	\geq \frac{M^2}{n \sum\limits_{j=1}^M \E_\lambda \I_j(\theta) + \J(\lambda)},
\end{equation}
where
\[
	\I_j(\theta)
	= \E_\theta \left( \frac{\partial \log \sfp_{g_\theta}(X)}{\partial \theta_j} \right)^2
\]
is the Fischer information of one observation, $p_{g_\theta}(x)$ is the density of $X$, and
\[
	\J(\lambda)
	= \int\limits_{\|\theta\|\leq h} 
	\sum\limits_{j=1}^M \left(\frac{\partial\lambda(\theta)}{\partial \theta_j}\right)^2 
	\frac{d\theta}{\lambda(\theta)}.
\]
First, let us bound $\J(\lambda)$.
\[
	\J(\lambda)
	= \sum\limits_{j=1}^M \int\limits_{\|\theta\| \leq h} \frac1{h^{2M}} \left( 
	\frac{\partial \lambda_0(\theta/h)}{\partial \theta_j} \right)^2 
	\frac{h^M}{\lambda_0(\theta/h)} d\theta
\]
Substituting $\theta = h \nu$, $\nu \in \B(0, 1)$, we obtain
\[
	\J(\lambda)
	= \sum\limits_{j=1}^M \int\limits_{\|\nu\| \leq 1} h^{-2} \left(\frac{\partial 
	\lambda_0(\nu)}{\partial \nu_j} \right)^2 \frac{d \nu}{\lambda_0(\nu)}.
\]
Since $\int\limits_{\|\nu\| \leq 1} \left(\frac{\partial \lambda_0(\nu)}{\partial \nu_j} \right)^2 
\frac{d \nu}{\lambda_0(\nu)}$ is finite, we conclude that
\[
	\J(\lambda) \lesssim \frac{M}{h^2}.
\]
Now, we focus on $\I_j(\theta)$:
\begin{align*}
	\log \sfp_{g_\theta}(x)
	&
	= \log \det\left( I_{d \times d} + h^{\beta - 1} \sum\limits_{j=1}^M \theta_j \nabla^2 \varphi \left(\frac{x - x_j}h \right) \right).
\end{align*}
For any $j \in \{1, \dots, M\}$ the partial derivative of $\log \sfp_{g_\theta}$ with respect to 
$\theta_j$ is equal to
\[
	\frac{\partial \log \sfp_{g_\theta}(x)}{\partial \theta_j}
	= h^{\beta - 1} \text{Tr} \left( A_\theta(x)^{-1} \nabla^2 \varphi \left(\frac{x - x_j}h \right) 
	\right),
\]
where
\[
	A_\theta(x) = I_{d \times d} + \sum\limits_{j=1}^M \theta_j h^{\beta - 1} \nabla^2 \varphi \left(\frac{x - x_j}h \right).
\]
Consider
\begin{align*}
	&
	\left| \text{Tr} \left( A_\theta(x)^{-1} \nabla^2 \varphi \left(\frac{x - x_j}h \right) \right) 
	\right|
	\\&
	= \left| \text{Tr} \left( \left( I_{d\times d} + h^{\beta - 1} \sum\limits_{k=1}^M \theta_k 
	\nabla^2 \varphi \left(\frac{x - x_k}h \right) \right)^{-1} \nabla^2 \varphi \left(\frac{x - x_j}h \right) \right) \right|.
\end{align*}
If $x \notin \B(x_j, h)$ then $\nabla^2 \varphi ((x - x_j) / h) = O_{d \times d}$ and, consequently, $\text{Tr} ( 
A_\theta(x)^{-1} \nabla^2 \varphi ((x - x_j) / h)) = 0$.
Otherwise, we have $\nabla^2 \varphi ((x - x_k) / h) = O_{d \times d}$ for all $k \neq j$.
Hence,
\begin{align*}
	&
	\left| \text{Tr} \left( \left( I_{d\times d} + h^{\beta - 1} \sum\limits_{k=1}^M \theta_k 
	\nabla^2 \varphi  \left(\frac{x - x_k}h \right) \right)^{-1} \nabla^2 \varphi \left(\frac{x - x_j}h \right) \right) \right|
	\\&
	= \left| \text{Tr} \left( \left( I_{d\times d} + h^{\beta - 1} \theta_j
	\nabla^2 \varphi \left(\frac{x - x_j}h \right) \right)^{-1} \nabla^2 \varphi \left(\frac{x - x_j}h \right) \right) \right|
	\\&
	= \left| \text{Tr} \left( \sum\limits_{k=0}^\infty (h^{\beta - 1} \theta_j)^k
	\left(\nabla^2 \varphi \left(\frac{x - x_j}h \right) \right)^{k+1} 
	\right) \right|
	\\&
	\lesssim \left| \text{Tr} \left( \nabla^2 \varphi \left(\frac{x - x_j}h \right) \right) \right|
	\lesssim d.
\end{align*}
Here we used the fact that $h^{\beta - 1} |\theta_j| \leq h^\beta$ is small, provided that 
the sample size is large enough.
Consider
\[
	\I_j(\theta)
	= \int\limits_{\Xset} \left( h^{\beta - 1} \text{Tr} \left( A_\theta(x)^{-1} \nabla^2 \varphi 
	\left(\frac{x - x_j}h \right) \right) \right)^2 \sfp_{g_\theta}(x) \rmd x.
\]
Since $\varphi$ is supported on $\B(0, 1)$, the last expression is equal to
\begin{align*}
	\I_j(\theta)
	&
	= h^{2\beta - 2} \int\limits_{\B(0, h)} \left( \text{Tr} \left( A_\theta(x)^{-1} 
	\nabla^2 \varphi \left(\frac{x - x_j}h \right) \right) \right)^2 \sfp_{g_\theta}(x) \rmd x 
	\\&
	\lesssim h^{2\beta - 2} d^2 \int\limits_{\B(0, h)} \sfp_{g_\theta}(x) \rmd x
	\lesssim h^{2\beta + d - 2} d^2.
\end{align*}
Here the last inequality follows from \eqref{eq:g_theta_derivative} and \Cref{lem:pmin_pmax}.
Thus, it holds that
\[
    \sup\limits_{\|\theta\| \leq h} \E_\theta \|\widehat \theta - \theta\|^2
    \geq \E_{\lambda} \E_{\theta} \|\widehat\theta - \theta\|^2
    \geq \frac{M^2}{n \sum\limits_{j=1}^M \E_\lambda \I_j(\theta) + \J(\lambda)}
    \gtrsim \frac{M h^2}{n h^{2\beta + d} d^2 + 1}.
\]

\subsection{Proof of Lemma \ref{lem:log_density}}
\label{sec:log_density_proof}
Recall that
\begin{align*}
	\log \sfp_{g_\theta}(x)
	&
	= \log \det\left( I_{d \times d} + h^{\beta - 1} \sum\limits_{j=1}^M \theta_j
	\nabla^2 \varphi \left(\frac{x - x_j}h \right) \right).
\end{align*}
Then, for any $j \in \{1, \dots, M\}$ the partial 
derivative of $\log \sfp_{g_\theta}$ with respect to $\theta_j$ is equal to
\[
	\frac{\partial \log \sfp_{g_\theta}(x)}{\partial \theta_j}
	= h^{\beta - 1} \text{Tr} \left( A_\theta(x)^{-1} \nabla^2 \varphi \left(\frac{x - x_j}h \right) 
	\right),
\]
where
\[
	A_\theta(x) =  I_{d \times d} + \sum\limits_{j=1}^M \theta_j h^{\beta - 1} \nabla^2 \varphi \left(\frac{x - x_j}h \right).
\]
Since, for any $j \in \{1, \dots, J\}$, $\psi_j$ is supported on $\B(0, 1)$ and $\{x_1, \dots, 
x_M\}$ is a $(2h)$-packing on $\Xset$, then, for any $j \in \{1, \dots, J\}$ and for any $x 
\in \B(x_j, h/2)$,
\[
	A_\theta(x) =  I_{d \times d} +  \theta_j h^{\beta - 1} \nabla^2 \varphi \left(\frac{x - 
	x_j}h \right).
\]
Consider $\text{Tr} \left( A_\theta(x)^{-1} \nabla^2 \varphi \left((x - x_j) / h \right) \right)$.
We recall that, by the construction,
\[
	-\nabla^2 \varphi (0)
	\succeq I_{d \times d}.
\]
Since $\varphi \in \H^{\beta+2}(\Xset, L)$, $\beta > 1$ and $r_0$ is such that $r_0 \leq h / (2 H_{\varphi})$,
we have
\[
	- \nabla^2 \varphi \left(\frac{x - x_j}h \right)
	\succeq 0.5 I_{d \times d}, \quad \text{for all } x \in \B(x_j, r_0).
\]
Taking into account that
\[
	\left\| \nabla^2 \varphi \left(\frac{x - x_j}h \right) \right\|
	\leq H_{\varphi}, 
\]
we obtain
\[
	\left\|  \theta_j h^{\beta - 1} \nabla^2 \varphi \left(\frac{x - x_j}h \right) \right\|
	\leq H_{\varphi} h^\beta < 1
	\quad \text{for any $h <  H_{\varphi}^{-1/\beta}$}.
\]
Then
\begin{align*}
	&
	\left( I_{d\times d} + \theta_j h^{\beta - 1} \nabla g_0 (g_0^{-1}(x))  
	\nabla \psi_j \left(\frac{x - x_j}h \right) \right)^{-1} 
	\\&
	= \sum\limits_{k=0}^\infty \theta_j^k  h^{k\beta - k} \left( -\nabla g_0 (g_0^{-1}(x))  
	\nabla \psi_j \left(\frac{x - x_j}h \right) \right)^k.
\end{align*}
This implies
\begin{align*}
	&
	- \text{Tr} \left( \left( I_{d\times d} + \theta_j h^{\beta - 1} \nabla^2 \varphi \left(\frac{x - x_j}h \right) \right)^{-1} \nabla^2 \varphi \left(\frac{x - x_j}h \right) \right)
	\\&
	= \text{Tr} \left( \sum\limits_{k=0}^\infty \theta_j^k  h^{k\beta - k}
	\left( -\nabla^2 \varphi \left(\frac{x - x_j}h \right) \right)^{k+1} \right)
	\\&
	\geq  \sum\limits_{k=0}^\infty \frac{\theta_j^k h^{k\beta - k}}{2^{k + 1}} \text{Tr} (I_{d \times d})
	\geq d/2.
\end{align*}
Then
\[
	\left| \frac{\partial \log \sfp_{g_\theta}(x)}{\partial \theta_j} \right|
	= h^{\beta - 1} \left| \text{Tr} \left( A_\theta(x)^{-1} \nabla^2 \varphi \left(\frac{x - x_j}h 
	\right) \right) \right|
	\geq h^{\beta - 1}d/2
	\quad \text{for all $x \in \B(x_j, r_0)$.}
\]

\end{document}